\documentclass[onecolumn]{IEEEtran}
\usepackage{setspace}
\makeatletter
    \DeclareRobustCommand{\qed}{%
      \ifmmode 
      \else \leavevmode\unskip\penalty9999 \hbox{}\nobreak\hfill
      \fi
      \quad\hbox{\qedsymbol}}
    \newcommand{\openbox}{\leavevmode
      \hbox to.77778em{%
      \hfil\vrule
      \vbox to.675em{\hrule width.6em\vfil\hrule}%
      \vrule\hfil}}
    \newcommand{\qedsymbol}{\openbox}
    \newenvironment{proof}[1][\proofname]{\par
      \normalfont
      \topsep6\p@\@plus6\p@ \trivlist
      \item[\hskip\labelsep\itshape
        #1.]\ignorespaces
    }{%
      \qed\endtrivlist
    }
    \newcommand{\proofname}{Proof}
\makeatother
\usepackage{cite, amssymb, mathtools, amsfonts, graphicx, textcomp, nicefrac, algorithm, algorithmic, siunitx, balance, setspace, enumitem}
\usepackage{tikz}
    \usetikzlibrary{positioning,calc}
\usepackage{pgfplots}
    \pgfplotsset{compat=1.18}
\allowdisplaybreaks
\newtheorem{lemma}{Lemma}
\newtheorem{theorem}{Theoerm}
\newtheorem{corollary}{Corollary}
\newtheorem{proposition}{Proposition}

\newtheorem{definition}{Definition}

\newtheorem{remark}{Remark}
\DeclareMathOperator{\vspan}{span}
\DeclareMathOperator{\diag}{\mathrm{diag}}
\begin{document}

\title{Dynamic Mode Decomposition with Control Liouville Operators}

\author{Joel A. Rosenfeld and Rushikesh Kamalapurkar
	\thanks{This research was supported in part by the Air Force Office of Scientific Research under award numbers FA9550-20-1-0127 and FA9550-21-1-0134, and the National Science Foundation (NSF) under award numbers 2027976 and 2027999. Any opinions, findings, or conclusions in this paper are those of the author(s) and do not necessarily reflect the views of the sponsoring agencies.}%
	\thanks{Joel A. Rosenfeld is with the Department of Mathematics and Statistics, University of South Florida, Tampa, FL 33620 USA (e-mail: rosenfeldj@usf.edu).}%
	\thanks{Rushikesh Kamalapurkar is with the School of Mechanical and Aerospace Engineering, Oklahoma State University, Stillwater, OK 74078 USA (e-mail: rushikesh.kamalapurkar@okstate.edu)}%
}

\maketitle

\begin{abstract}
    This paper builds the theoretical foundations for dynamic mode decomposition (DMD) of control-affine dynamical systems by leveraging the theory of vector-valued reproducing kernel Hilbert spaces (RKHSs). Specifically, control Liouville operators and control occupation kernels are introduced to separate the drift dynamics from the input dynamics. A given feedback controller is represented through a multiplication operator and a composition of the control Liouville operator and the multiplication operator is used to express the nonlinear closed-loop system as a linear total derivative operator on RKHSs. A spectral decomposition of a finite-rank representation of the total derivative operator yields a DMD of the closed-loop system. The DMD generates a model that can be used to predict the trajectories of the closed-loop system. For a large class of systems, the total derivative operator is shown to be compact provided the domain and the range RKHSs are selected appropriately. The sequence of models, resulting from increasing-rank finite-rank representations of the compact total derivative operator, are shown to converge to the true system dynamics, provided sufficiently rich data are available. Numerical experiments are included to demonstrate the efficacy of the developed technique.
\end{abstract}


\section{Introduction}
\IEEEPARstart{S}{pectral} methods for identification of nonlinear systems utilize representations of unknown, finite-dimensional nonlinear dynamics, in discrete or continuous time, as linear operators over infinite dimensional spaces (cf. \cite{SCC.Mezic2005}). In the discrete-time case, this linear operator is a composition operator called the Koopman operator \cite{SCC.Koopman1931}. In the continuous time case, it is a total derivative operator called the Liouville operator \cite{SCC.Rosenfeld.Kamalapurkar.ea2022} (or the Koopman generator, in special cases where it can be obtained as the limit of a sequence of Koopman operators with decreasing sample times \cite[Section 7.5]{SCC.Lasota.Mackey1994}). In dynamic mode decomposition (DMD), trajectories of a dynamical system are used to construct a finite-rank representation of the aforementioned linear operator \cite{SCC.Kutz.Brunton.ea2016}. The finite-rank representation is then diagonalized and the resultant eigenfunction and eigenvalues are used to provide a representation of the identity function. This representation provides the dynamic modes of the system as vector-valued coefficients attached to the eigenfunctions. Thereafter, a state trajectory can be predicted as a sum of exponential functions multiplied by the dynamic modes (cf. \cite{SCC.Williams.Rowley.ea2015,SCC.Kutz.Brunton.ea2016,SCC.Rosenfeld.Kamalapurkar.ea2022}). 

The primary application area of Koopman spectral analysis of dynamical systems has been fluid dynamics, where DMD is compared with proper orthogonal decomposition (POD) for nonlinear fluid equations (cf. \cite{SCC.Mezic2013}). DMD has also been employed in the study of stability properties of dynamical systems \cite{SCC.Vaidya.Mehta.ea2010,SCC.Mauroy.Mezic2016}, neuroscience \cite{SCC.Brunton.Johnson.ea2016}, financial trading \cite{SCC.Mann.Kutz2016}, feedback stabilization \cite{SCC.Huang.Ma.ea2018}, optimal control \cite{SCC.Sootla.Mauroy.ea2018}, modeling of dynamical systems \cite{SCC.Proctor.Brunton.ea2016,SCC.Quade.Abel.ea2018,SCC.Sinha.Huang.ea2019}, and model-predictive control \cite{SCC.Arbabi.Korda.ea2018}. For a generalized treatment of DMD as a Markov model, see \cite{SCC.Jayaraman.Lu.ea2019}.

Extensions of the idea of Koopman operator-based DMD to systems with control can be loosely categorized in three categories: spectral analysis of the drift (zero-input) dynamics \cite{SCC.Surana2016}, input-parameterized Koopman operators \cite{SCC.Proctor.Brunton.ea2018}, and reformulation as an autonomous state-control dynamical system \cite{SCC.Korda.Mezic2018a}. These methods rely on discretization of continuous-time systems, either for computation (when Koopman operators are used), or for analysis (when Koopman generators are used), and as such, are only applicable to systems that admit a globally well-defined discretization (i.e., systems that cannot escape to infinity in finite time starting from any initial condition). When dealing with Koopman generators, the data required for a spectral decomposition typically include the time derivative of the state, which is not generally available. Recently, inspired by the notion of occupation measures \cite{SCC.Lasserre.Henrion.ea2008} defined on Banach spaces of continuous functions, the authors in \cite{SCC.Rosenfeld.Russo.eatoappear} defined analogous objects on reproducing kernel Hilbert spaces (RKHSs). The so-called occupation kernels, when combined with operators such as the Liouville operator, provide a method for spectral analysis of continuous-time systems directly, without the need for discretization.

The paradigm shift afforded by occupation kernels arises through the consideration of the state trajectory as the fundamental unit of data \cite{SCC.Rosenfeld.Russo.eatoappear}. This paper, along with the preliminary results reported in \cite{SCC.Rosenfeld.Kamalapurkar2021}, build on the foundations developed in \cite{SCC.Rosenfeld.Russo.eatoappear} to address DMD of control-affine dynamical systems. To address systems with control, the occupation kernels are augmented by the control signals, resulting in the so-called \emph{control occupation kernels}, and the Liouville operator is extended to include the input dynamics, to yield the so-called \emph{control Liouville operator} \cite{SCC.Rosenfeld.Kamalapurkar2021}. The extension in \cite{SCC.Rosenfeld.Kamalapurkar2021} utilizes the theory of vector-valued RKHSs (vvRKHSs), introduced in \cite{SCC.Pedrick1957} and \cite{SCC.Schwartz1964},  and extensively studied in a machine learning context in \cite{SCC.Micchelli.Pontil2005}, \cite{SCC.Carmeli.DeVito.ea2006}, and \cite{SCC.Carmeli.DeVito.ea2010}. Multiplication operators that map between scalar-valued and vector-valued RKHSs are also utilized to define a total derivative operator that represents the dynamics of the closed-loop system controlled using a feedback controller. Using the control occupation kernels, the control Liouville operators, and the multiplication operators, a technique for discretization-free DMD of control-affine, continuous-time, nonlinear systems is developed. The developed control-Liouville DMD (CLDMD) and singular control-Liouville DMD (SCLDMD) methods yield a predictor that can predict the closed-loop behavior of a system under any given locally Lipschitz continuous feedback controller by measuring its response to different open-loop control signals.

The definitions of control occupation kernels and control Liouville operators used in this paper were first reported in the conference paper \cite{SCC.Rosenfeld.Kamalapurkar2021}. In that paper, a finite-rank representation of the closed-loop total derivative operator is indirectly derived through its adjoint. In this paper, a finite-rank representation of the closed-loop total differential operator is obtained directly, resulting in a simpler DMD algorithm. Furthermore, this paper includes a novel singular value decomposition (SVD)-based finite-rank representation of a new total derivative operator which converges in norm to the true operator with increasing rank under a compactness assumption and given sufficiently rich data. Examples of general classes of nonlinear systems where the new total derivative operators are compact are also provided to justify the compactness assumptions.

The paper is structured as follows. Section \ref{sec:Problem Formulation} formulates the prediction problem. Section \ref{sec:OperatorDMD} summarizes the overall approach. Section \ref{sec:vvRKHS} introduces the concept of vvRKHSs. Section \ref{sec:OperatorRepresentation} introduces the control occupation kernels, the control Liouville operators, and the multiplication operators needed to develop a representation of a closed-loop nonlinear system in terms of linear operators on a set of Hilbert spaces. Section \ref{sec:SVD-DMD} introduces an SVD-based approach to DMD. Section \ref{sec:eig-DMD} introduces an eigendecomposition-based approach to DMD. Section \ref{sec:computations} introduces the computational tools required to generate the finite-rank representations. Section \ref{sec:sims} presents numerical experiments to validate the developed technique. Section \ref{sec:discussion} discusses the results of the numerical experiments, and Section \ref{sec:Conclusion} concludes the paper.

\section{Problem Statement\label{sec:Problem Formulation}}
Given Carath\'{e}odory solutions $\{ \gamma_{u_i}:[0,T_i] \to \mathbb{R}^n \}_{i=1}^M$ of a nonlinear control-affine system of the form
\begin{equation}
    \dot{x} = f(x) + g(x) u_i(t), \quad x(0) = \gamma_{u_i}(0) \label{eq:openloop_dynamics}
\end{equation}
under Lebesgue measurable, bounded control inputs $\{u_i : [0,T_i] \to \mathbb{R}^m\}_{i=1}^M$, the objective of this paper is to provide an operator theoretic approach for the analysis of the \emph{closed loop} system
\begin{equation}
    \dot{x} = f(x) + g(x)\mu(x) \eqqcolon F_\mu (x),\label{eq:control-affine}
\end{equation}
where $x\in\mathbb{R}^n$ is the state, $\mu : \mathbb{R}^n \to \mathbb{R}^m$ is a locally Lipschitz continuous feedback controller, $f:\mathbb{R}^n \to \mathbb{R}^n$ and $g: \mathbb{R}^n \to \mathbb{R}^{n\times m}$ are locally Lipschitz continuous functions corresponding to the drift dynamics and the control effectiveness matrix, respectively, and $\dot{x}$ denotes the time derivative of $x$. The observed control trajectories and control inputs will allow for the construction of a finite-rank representation of the so-called \emph{control Liouville operator}, which is a generalization of the Liouville operator introduced in \cite{SCC.Rosenfeld.Kamalapurkar.ea2022}.

Similar to the robot manipulator examples in \cite{SCC.Kamalapurkar.Walters.ea2018} most Euler-Lagrange systems with invertible inertia matrices can be expressed in the control-affine form. The Euler-Lagrange equations are used to describe a large class of physical systems (cf. \cite{SCC.Goldstein.Poole.ea2002}), and as such, various methods for control and identification of nonlinear systems in the Euler-Lagrange form have been studied in detail over the years (see, e.g., \cite{SCC.Ortega.Loria.ea1998,SCC.Morabito.Teel.ea2004,SCC.Feng.Hu.ea2018}).  Since most physical systems of practical importance such as robot manipulators \cite{SCC.Behal.Dixon.ea2009} and ground, air, and maritime vehicles and vessels \cite{SCC.Kamalapurkar.Walters.ea2018} have inertia matrices that are invertible over large operating regions, control-affine models encompass a large class of physical systems.

\section{Operators and Dynamic Mode Decomposition\label{sec:OperatorDMD}}
In this section, the general idea behind the developed operator-theoretic DMD approach is introduced. The approach relies on representation of a closed loop dynamical system as an operator that maps between suitable function spaces. For the motivational discussion in this section, assume that given functions $f$, $g$, and $\mu$, and RKHSs $\tilde{H}_d$ and $\tilde{H}_r $ defined on a compact set $X\subset\mathbb{R}^n$, there exist a set, $\mathcal{D}\left(A_{F_{\mu}}\right)\subset \tilde{H}_d$ and a \emph{total derivative operator} $A_{F_{\mu}}:\mathcal{D}\left(A_{F_{\mu}}\right)\to\tilde{H}_r$ such that
\begin{enumerate}[label={(R\arabic*)},leftmargin=*]
    \item for all $h \in \mathcal{D}\left(A_{F_{\mu}}\right)$, $A_{F_{\mu}}h\coloneqq\frac{\partial h}{\partial x}  F_{\mu} \in \tilde{H}_r$, where $\frac{\partial h}{\partial x}$ is a row vector, and\label{req:R1}
    \item $ h_{\mathrm{id},j} \in \mathcal{D}\left(A_{F_{\mu}}\right) $ for all $j=1,\cdots,n$, where $h_{\mathrm{id}} = \begin{bmatrix} h_{\mathrm{id},1},&\cdots&,h_{\mathrm{id},n} \end{bmatrix}^{\top}$ is the identity function, with components defined as $h_{\mathrm{id},j}(x) = x_j$ for all $x \in X$.\label{req:R2}
\end{enumerate}
Note that the total derivative operator $A_{F_{\mu}}$ above is the Liouville operator (or the Koopman generator) with symbol $F_{\mu} = f+g\mu$ as defined in \cite{SCC.Rosenfeld.Kamalapurkar.ea2022}. As such, a DMD of the closed loop system could be obtained using the methods presented in \cite{SCC.Rosenfeld.Kamalapurkar.ea2022} \emph{provided data generated by the closed loop system $\dot{x} = F_{\mu}(x)$ is available}. The objective in this paper is to develop a model of the system using a feedback-agnostic data set. That is, given any feedback controller $\mu:\mathbb{R}^n\to\mathbb{R}^m$ and a data set recorded by exciting the open-loop system $\dot{x} = f(x) + g(x)u$ using control signals $u = u_i:[0,T_i]\to\mathbb{R}^n$, $i=1,\ldots,M$, we aim to build a predictive model of the closed loop system $\dot{x} = F_{\mu}(x)$.

\subsection{The Eigendecomposition Approach}
If $\tilde{H}_d = \tilde{H}_r $, $\phi$ is an eigenfunction of $A_{F_{\mu}}$ with eigenvalue $\lambda$, and $\gamma_{\mu}$ is a controlled trajectory arising from \eqref{eq:control-affine}, then it follows that
{\medmuskip=0mu \thinmuskip=0mu \thickmuskip=0mu
\begin{multline*}
    \frac{\mathrm{d} \left(\phi\left(\gamma_{\mu}\left(t\right)\right)\right)}{\mathrm{d}t}  = \frac{\partial \phi}{\partial x} \left(\gamma_{\mu}\left(t\right)\right) \Big(f\left(\gamma_{\mu}\left(t\right)\right) + g\left(\gamma_{\mu}\left(t\right)\right) \mu\left(\gamma_{\mu}\left(t\right)\right)\Big)\\
    = \left[A_{F_{\mu}} \phi\right]\left(\gamma_{\mu}\left(t\right)\right) = \lambda \phi\left(\gamma_{\mu}\left(t\right)\right).
\end{multline*}}
Hence, $\phi(\gamma_{\mu}(t)) = \mathrm{e}^{\lambda t} \phi(\gamma_{\mu}(0)).$

If the the span of the eigenfunctions $\{\phi_i\}_{i=1}^{\infty}$ of $A_{F_{\mu}}$ is dense in $\tilde{H}_d$, then the identity function can be decomposed using the eigenfunctions as $h_{\mathrm{id}}(x) = \lim_{M\to\infty}\sum_{i=1}^M \xi_{i,M} \phi_i(x)$, where $\xi_{i,M} \in \mathbb{C}^n$ are the \emph{dynamic modes} of the closed loop system. Moreover, it follows that
\begin{equation}
    \gamma_{\mu}(t) = h_{\mathrm{id}}(\gamma_\mu(t)) = \lim_{M\to\infty}\sum_{i=1}^M  \xi_{i,M} \phi_i(\gamma_\mu(0)) \mathrm{e}^{\lambda_i t},\label{eq:infinite_spectral_reconstruction}
\end{equation}
where $\lambda_i$ denotes the eigenvalue corresponding to the eigenfunction $\phi_i$, and the coefficients $\xi_{i,M}$ depend on $M$ because the eigenfunctions are not generally orthogonal.

If the eigenfunctions, the eigenvalues, and the modes could be computed from data, then a finite truncation of \eqref{eq:infinite_spectral_reconstruction} could be used as a predictive model. However, since the operator $A_{F_{\mu}}$ cannot generally be expected to be bounded, even the existence of eigenfunctions cannot be guaranteed.

The idea in DMD is to construct a finite rank (say rank $M$) approximation (say $\hat A_{F_{\mu},M}$) of $A_{F_{\mu}}$. Then, the eigenfunctions $\{\hat\phi_{i,M}\}_{i=1}^{M}$, the eigenvalues $\{\hat\lambda_{i,M}\}_{i=1}^{M}$, and the modes $\{\hat\xi_{i,M}\}_{i=1}^{M}$ of $\hat A_{F_{\mu},M}$ are computed and used as proxies in a finite truncation of \eqref{eq:infinite_spectral_reconstruction} to generate a predictive model.

If the operators $\hat A_{F_{\mu},M}$ can be shown to converge to $A_{F_{\mu}}$ \emph{in the norm topology}, then given any $\epsilon > 0$, there exists $M$ such that for all $i=1,\ldots,M$, the pairs $(\hat\phi_{i,M},\hat\lambda_{i,M})$ are approximate eigenpairs for the true operator $A_{F_{\mu}}$. That is, for all $i=1,\ldots,M$ and for all $x\in X$, $\left\vert \left[A_{F_{\mu}} \hat\phi_{i,M}\right](x) - \hat\lambda_{i,M} \hat\phi_{i,M}(x) \right\vert < \epsilon $. The approximate eigenpairs can then be used to obtain a model that, given rich enough data and a large enough $M$, can accurately predict the system trajectories in $X$ over a finite horizon.

While requirements \ref{req:R1} and \ref{req:R2} above, compactness of the Liouville operator, and density of the eigenfunctions in $\tilde{H}_d$ are difficult to guarantee in general, empirical evidence suggests that the eigenfunctions $\hat\phi_{i,M}$ are expressive enough to approximate $h_{\mathrm{id},i}$ in a variety of applications \cite{SCC.Rosenfeld.Kamalapurkar.ea2022}. Since $\hat\phi_{i,M}$ are computed as linear combinations of reproducing kernels or occupation kernels, the empirical evidence could be explained by the postulate that the approximate eigenfunctions inherit universality properties of the reproducing kernels and the occupation kernels \cite{SCC.Rosenfeld.Russo.eatoappear}. A theoretical examination of the expressiveness of the approximate eigenfunctions for a specific operators, Hilbert spaces, and data set is out of the scope of this article. 

Convergence of the finite-rank representation to the true operator in the norm topology is also typically impossible to guarantee in the eigendecomposition-based DMD framework \cite{SCC.Korda.Mezic2018a, SCC.Rosenfeld.Kamalapurkar.ea2022}. As such, similar to most DMD techniques, the eigendecomposition approach, while well-motivated by the theory presented in this paper, is a heuristic technique. On the other hand, as shown in \cite{SCC.Rosenfeld.Kamalapurkar2023a}, obtaining norm convergence of finite rank representations to the true Liouville operator is possible in an SVD-based framework. 

\subsection{The Singular Value Decomposition Approach}
In the SVD-based framework, two different RKHSs $\tilde{H}_d$ and $\tilde{H}_r$ are selected as the domain and the co-domain of $A_{F_{\mu}}$, respectively. If the domain and the range RKHSs are selected carefully, then for a large class of nonlinear systems, the operator $A_{F_{\mu}}$ can be shown to be compact. Compactness trivially ensures satisfaction of Requirement \ref{req:R1} above. Requirement \ref{req:R2} can be met by proper selection of $\tilde{H}_d$ (see Section \ref{sec:SVD-DMD}). Compactness also allows for the construction of the needed sequence $\hat{A}_{F_{\mu},M}$ that converges to $A_{F_{\mu}}$ in the norm topology. The left and right singular functions of $\hat{A}_{F_{\mu},M}$ can then be used to generate a sequence of system models that converges to the true system model.

In particular, the closed-loop model $\dot{x} = f(x) + g(x) \mu(x)$ can be expressed in terms of the total derivative operator as
\begin{equation}\label{eq:total_derivative_model}
    \dot{x} = \frac{\partial h_{\mathrm{id}}}{\partial x}(x)\begin{bmatrix}
        f(x) & g(x)
    \end{bmatrix}\begin{bmatrix}1\\\mu(x)\end{bmatrix} = [A_{F_{\mu}} h_{\mathrm{id}}] (x),
\end{equation} 
where the notation $ A_{F_{\mu}} h_{\mathrm{id}}$ is used to denote the operator $A_{F_{\mu}}$ acting on every row of the vector-valued function $h_{\mathrm{id}}$. If $A_{F_{\mu}}:\tilde{H}_d\to\tilde{H}_r$ is a compact operator, then there exist singular values $ \left\{\sigma_i\right\}_{i=1}^\infty\subset \mathbb{R}$, left singular functions $\left\{\phi_i\right\}_{i=1}^\infty\subset{\tilde{H}_d}$, and right singular functions $\left\{\psi_i\right\}_{i=1}^\infty\subset \tilde{H}_r$ such that
\begin{equation}\label{eq:infinite_spectral_reconstruction_svd}
    \dot{x} = \sum_{i=1}^\infty \sigma_i \left\langle h_{\mathrm{id}},\phi_i\right\rangle_{\tilde{H}_{d}} \psi_i(x),
\end{equation}
where the notation $\left\langle h_{\mathrm{id}},\phi_i\right\rangle_{\tilde{H}_{d}}$ is used to denote the $n-$vector $\begin{bmatrix}\left\langle h_{\mathrm{id},1},\phi_i\right\rangle_{\tilde{H}_{d}},&\ldots,&\left\langle h_{\mathrm{id},n},\phi_i\right\rangle_{\tilde{H}_{d}}\end{bmatrix}^\top$. The idea in singular DMD is to use the SVD of $\hat A_{F_{\mu},M}$ as a proxy in a finite truncation of \eqref{eq:infinite_spectral_reconstruction_svd} to construct a predictive model.

\subsection{Related Work}
Operator-based DMD methods for systems with control can be loosely categorized in three categories: spectral analysis of the drift (zero-input) dynamics \cite{SCC.Surana2016}, input-parameterized Koopman operators \cite{SCC.Proctor.Brunton.ea2018}, and reformulation as an autonomous state-control dynamical system \cite{SCC.Korda.Mezic2018a}.

If data can be collected for the system with zero inputs, or if the system is affine in control, then techniques such as dynamic mode decomposition with control (DMDc) \cite{SCC.Proctor.Brunton.ea2016}, sparse nonlinear system identification with control (SINDYc) \cite{SCC.Brunton.Proctor.ea2016}, extended dynamic mode decomposition with control (EDMDc) \cite{SCC.Korda.Mezic2018a}, bilinearization \cite{SCC.Goswami.Paley2022}, etc., can be utilized to estimate eigenvalues and eigenfunctions of the Koopman operator, or the Koopman generator, of the drift (zero-input) dynamics. In the case of control-affine systems, the eigenvalues and eigenfunctions can also be utilized to solve a wide variety of problems including, but not limited to, reachability \cite{SCC.Goswami.Paley2022}, optimal control \cite{SCC.Kaiser.Kutz.ea2021}, model-based predictive control, \cite{SCC.Folkestad.Pastor.ea2020}, and observer synthesis \cite{SCC.Surana2016}.     

A different approach to operator theoretic analysis of systems with control is via input-parameterized Koopman operators \cite{SCC.Proctor.Brunton.ea2018}. The central idea in this family of methods is that if the input is constant, then the dynamical system is autonomous, and as such, admits a Koopman operator. Given a set of possible input levels, an input-parameterized family of Koopman operators (or generators) can thus be constructed \cite{SCC.Kaiser.Kutz.ea2021}. This observation is particularly useful when utilized for spectral analysis of control-affine systems, where Koopman generators are themselves affine in control. The state of the system can thus be predicted using a linear combination of a finite number of input-parameterized Koopman generators \cite{SCC.Peitz.Otto.ea2020}. In addition to motivating DMDc and EDMDc, input-parameterized Koopman generators can also be used for various control and estimation tasks \cite{SCC.Otto.Rowley2021}.

Systems with control can also be analyzed by studying operators that operate on a more general set of observables. Instead of observables that are functions of the state in the typical Koopman framework, the observables here are functions of the state and the control \cite{SCC.Korda.Mezic2018a,SCC.Proctor.Brunton.ea2018}. The methods in this category include Koopman with inputs and control (KIC) \cite{SCC.Proctor.Brunton.ea2018} and linear and bilinear predictors \cite{SCC.Korda.Mezic2018a}. The KIC approach is cogent if the control signal itself is produced by a dynamical system, and leads to useful heuristics when it is not. In \cite{SCC.Korda.Mezic2018a}, the shift operator is used as the dynamics of the control signal to develop a Koopman operator that operates on observables defined on an infinite-dimensional state space that includes the space of all possible control \emph{sequences}. Spectral analysis of this operator with carefully selected observables yields linear and bilinear predictors for the underlying nonlinear system. Applications of this approach include model-based predictive control \cite{SCC.Korda.Mezic2018a}, robust model-based predictive control \cite{SCC.Zhang.Pan.ea2022}, and system identification \cite{SCC.Mauroy.Goncalves2019}.

In this paper, the operator $A_{F_{\mu}}$ is constructed as a composition of two operators, a differential operator that maps from $\tilde{H}_d$ into a vvRKHS and a multiplication operator that maps from the vvRKHS either back into $\tilde{H}_{d}$ (the eigendecomposition approach) or into $\tilde{H}_r$ (the SVD approach) (see Fig. \ref{fig:operator-diagram}).
\section{Vector-valued Reproducing Kernel Hilbert Spaces\label{sec:vvRKHS}}
In this section, properties of vvRKHSs relevant to topic under consideration are reviewed. The review relies heavily on the discussion given in \cite{SCC.Carmeli.DeVito.ea2010}.
\begin{definition}
    Let $\mathcal{Y}$ be a Hilbert space, and let $H$ be a Hilbert space of functions from a set $X$ to $\mathcal{Y}$. The Hilbert space $H$ is a \emph{vvRKHS} if for every $v \in \mathcal{Y}$ and $x \in X$, the functional $f \mapsto \langle f(x), v \rangle_{\mathcal{Y}}$ is bounded.
\end{definition}
A vvRKHS is a direct generalization of a ``scalar-valued'' RKHS, since for a fixed $v \in \mathcal{Y}$, the collection of functions $\{ g(x) = \langle f(x), v \rangle_{\mathcal{Y}} : f \in H \}$ forms an RKHS of scalar-valued functions. 

By the Riesz representation theorem, for each $x \in X$ and $v \in \mathcal{Y}$, there exists a unique function $K_{x,v} \in H$ such that $\langle f, K_{x,v} \rangle_H = \langle f(x), v \rangle_{\mathcal{Y}}$ for all $f \in H$. The fact that the mapping $v \mapsto K_{x,v}$ is linear over $\mathcal{Y}$ yields a linear operator $ K_x:\mathcal{Y}\to H $, defined as $ K_{x} \coloneqq v\mapsto K_{x,v} $, called \emph{the kernel centered at $x$, associated with $H$.} The operator $ K: X \times X \to \mathcal{L}(\mathcal{Y},\mathcal{Y}) $, defined as $ K(x,y) := K_x^*K_y $, where $K_x^*:H\to\mathcal{Y} $ is the adjoint of $K_x$ and $\mathcal{L}(\mathcal{Y},\mathcal{Y})$ is the space of linear operators from $\mathcal{Y}$ to $\mathcal{Y}$, is called the \emph{reproducing kernel of $ H $.} For any $f\in H$, $x\in X$, and $v\in\mathcal{Y}$, we have $\langle K_x^*f ,v\rangle_{\mathcal{Y}} = \langle f ,K_x v\rangle_{H} = \langle f(x) ,v\rangle_{\mathcal{Y}} $, and as a result, the reproducing property $K_x^* f = f(x)$. With $f = K_y v$, we see that for all $v\in \mathcal{Y}$, $ [K_y v](x) = K_x^* K_y v = K(x,y)v $.

In the particular case that $\mathcal{Y} = \mathbb{R}^n$, $K(x,y)$ is a real-valued $n \times n$ matrix for fixed $x,y \in X$. As a result one can construct several examples of vector-valued kernels. Indeed, given a scalar-valued RKHS $\tilde{H}$ over $X$, with the corresponding reproducing kernel $\tilde{K}:X\times X\to\mathbb{R}$, and a positive definite matrix, $A \in \mathbb{R}^{n\times n}$, the operator $ (x,y) \mapsto A \tilde{K}(x,y) $ that maps from $X\times X$ to $\mathcal{L}\left(\mathbb{R}^n,\mathbb{R}^n\right)$ is a reproducing kernel of a vvRKHS.

Similar to scalar-valued kernels, it can be shown that the span of vector-valued kernels is dense in $H$.
\begin{proposition}\label{prop:kernelDensity}
    The span of the set  $ E := \{ K_{x,v} : v \in \mathcal{Y} \text{ and } x \in X \}$, is dense in $H$.
\end{proposition}
\begin{proof}
    Suppose that $h \in E^{\perp}$, then given a fixed $x \in X$, $\langle h, K_{x,v} \rangle_{H} = \langle h(x), v \rangle_{\mathcal{Y}} = 0$ for all $v \in \mathcal{Y}$. Hence, $h(x) = 0 \in \mathcal{Y}$. Since $x$ was arbitrarily selected, $h \equiv 0 \in H$. Thus, $E^{\perp} = \{ 0 \}$ and $\overline{\vspan(E)} = (E^{\perp})^\perp = H$.
\end{proof}
As a consequence of Proposition \ref{prop:kernelDensity}, given $\epsilon > 0$ and $h \in H$, there is a finite linear combination of vector-valued kernels that approximate $h$ with an error smaller than $\epsilon$ in the Hilbert space norm.

In the following development, unless otherwise specified, it is assumed that $X\subset\mathbb{R}^n$ is compact, the Hilbert space $\mathcal{Y}$ is selected to be $\mathbb{C}^{1\times (m+1)}$ with the usual definitions of vector norms and inner products, $\tilde{H}_d$ and $\tilde{H}_r$ are RKHSs of continuously differentiable functions from $X$ to $\mathbb{C}$, and $H$ is a vvRKHS of continuous functions from $X$ to $\mathbb{C}^{1\times (m+1)}$. The reproducing kernel of $H$ is denoted by $K:X \times X \to \mathcal{L}(\mathbb{R}^{1\times (m+1)},\mathbb{R}^{1\times (m+1)})$ and the reproducing kernels of $\tilde{H}_d$ and $\tilde{H}_r$ are denoted by $\tilde{K}_d:X \times X \to\mathbb{R}$ and $\tilde{K}_r:X \times X \to\mathbb{R}$, respectively. When the domain and the range RKHSs are identical, the subscripts $d$ and $r$ are omitted. The Hilbert space  $\mathcal{Y}$ is selected to be a space of row vectors to accommodate the row vector convention for partial derivatives. As such, the linear operation of $K_{x}$ on $v \in \mathcal{Y}$ is expressed as $K_{x,v} = vK_{x}$.

\section{Closed Loop Nonlinear Systems as Operators over RKHSs\label{sec:OperatorRepresentation}}

To solve the problem as stated in Section \ref{sec:Problem Formulation} in a vvRKHS framework, the closed-loop nonlinear system is expressed in terms of operators over two RKHSs and a vvRKHS. A majority of the definitions and propositions in this section were first introduced in \cite{SCC.Rosenfeld.Kamalapurkar2021}. The definitions are included here for completeness and the proofs of most of the propositions are more detailed than the corresponding proofs in \cite{SCC.Rosenfeld.Kamalapurkar2021}.
\subsection{Control Liouville Operators and Multiplication Operators \label{subsec:controlLiouvilleOperators}}
Representation of a controlled system in terms of operators can be realized using the so-called control Liouville Operator.
\begin{definition}
    Let $f:\mathbb{R}^n \to \mathbb{R}^n$ and $g: \mathbb{R}^n \to \mathbb{R}^{n\times m}$ be locally Lipschitz continuous functions and the set 
    \begin{equation*}
        \mathcal{D}(A_{f,g}) \coloneqq \{ h \in \tilde{H}_d : x\mapsto\frac{\partial h}{\partial x} (x) \begin{bmatrix} f(x) & g(x) \end{bmatrix} \in H \}
    \end{equation*}
    be the domain of the operator, $A_{f,g} : \mathcal{D}(A_{f,g}) \to H$, given as
    \begin{equation*}
        \left[A_{f,g} h\right](x) \coloneqq \frac{\partial h}{\partial x} (x) \begin{bmatrix} f(x) & g(x) \end{bmatrix}.
    \end{equation*}
    The operator $A_{f,g}$ is called \emph{the control Liouville operator corresponding to $f$ and $g$ over $H$}.
\end{definition}
Control Liouville operators are a direct generalization of the more traditional Liouville operators, where the drift dynamics and control effectiveness components of the dynamics are separated on the operator theoretic level. Vector-valued RKHSs arise naturally in this context, where the partial derivative of $h \in \mathcal{D}(A_{f,g})$ with respect to $x$ is a row vector of dimension $n$, and through a dot product with $f$ and multiplication by the matrix $g$, the result of the operation of $A_{f,g}$ on $h$ is a row vector with dimension $m+1$.

The control Liouville operator does not depend on the control input, and as such, is not sufficient by itself for prediction of system behavior. An additional operator is thus required to complete the construction of the operator $A_{F_{\mu}}$ alluded to in Section \ref{sec:OperatorDMD}. The controller is incorporated in the developed framework via a multiplication operator. The inclusion of this multiplication operator, in addition to the newly defined control Liouville operator, sets the theoretical foundations of DMD of controlled systems apart from the uncontrolled case studied in \cite{SCC.Rosenfeld.Kamalapurkar.ea2022} and \cite{SCC.Rosenfeld.Kamalapurkar2023a}.
\begin{definition}
    For a continuous function $\nu:X\to\mathcal{Y}$, the \emph{multiplication operator with symbol $\nu$}, denoted by $M_{\nu} : \mathcal{D}(M_{\nu}) \to \tilde{H}_r$, is defined as
    \begin{equation*}
       \left[M_{\nu} h\right](\cdot) \coloneqq \langle h(\cdot), \nu(\cdot) \rangle_{\mathcal{Y}},
    \end{equation*}
    where $\mathcal{D}(M_{\nu}) := \{ h \in H : x\mapsto\langle h(x), \nu(x) \rangle_{\mathcal{Y}} \in \tilde{H}_r \}$.
\end{definition}

Given the continuous function $\overline{\mu}:\mathbb{R}^n\to\mathbb{R}^{1\times m+1}$ derived from a feedback controller $\mu : \mathbb{R}^n \to \mathbb{R}^m$ as $\overline{\mu}(x) \coloneqq \begin{bmatrix} 1 & \mu(x)^{\top} \end{bmatrix}$, the corresponding multiplication operator $M_{\overline{\mu}} : \mathcal{D}(M_{\overline{\mu}}) \to \tilde{H}_{r}$ is given as $M_{\overline{\mu}} h = x\mapsto h(x) \begin{bmatrix} 1 & \mu(x)^{\top} \end{bmatrix}^{\top} $ and $\mathcal{D}(M_{\overline{\mu}}) = \{ h \in H :x\mapsto h(x) \begin{bmatrix} 1 & \mu(x)^{\top} \end{bmatrix}^{\top} \in \tilde{H}_r \}.$
The operator $M_{\overline{\mu}} A_{f,g}$ maps from $\mathcal{D}(A_{f,g}) \subset \tilde{H}_d$ to $\tilde{H}_r$, and plays the role of the operator $A_{F_{\mu}}$ described in Section \ref{sec:OperatorDMD}. In the above construction, it is assumed that the image of $ A_{f,g} $ falls within the domain of $ M_{\overline{\mu}} $. This assumption is not easy to verify in general, but it is trivially met in the example presented in Section \ref{sec:SVD-DMD}, where the RKHSs are Bargmann-Fock spaces restricted to the set of real numbers. The control Liouville operator will be assumed to be compact in Section \ref{sec:SVD-DMD} and densely defined in Section \ref{sec:eig-DMD}. Further comments on the density and the compactness assumptions, including examples of systems and RKHSs for which these assumptions are met, are provided in the respective sections.
\begin{figure}
    \centering
\begin{tikzpicture}[>= latex]
    \node[minimum size=1.25cm, circle, fill=lightgray, circle,] (H) {$H$};
    \node[minimum size=1.25cm, fill=lightgray, circle, above left=1.5cm and 3cm of H.center, anchor=center] (Hd) {$\tilde{H}_d$};
    \node[minimum size=1.25cm, fill=lightgray, circle, below left=1.5cm and 3cm of H.center, anchor=center] (Hr) {$\tilde{H}_r$};
    \node[minimum size=0.75cm, draw, above right=1.5cm and 1cm of Hr.center, anchor=center] (X) {$X$};
    \node[minimum size=0.75cm, draw, above left=1.5cm and 1cm of Hr.center, anchor=center] (R) {$\mathbb{R}$};
    \node[minimum size=0.75cm, draw, right=2cm of H.center, anchor=center] (Y) {$\mathcal{Y}$};
    \draw [->,dashed,thick] plot [smooth] coordinates {(X.north) ($(Hd.center) - (0,0.25)$) (R.north)};
    \draw [->,dashed,thick] plot [smooth] coordinates {(X.south) ($(Hr.center) + (0,0.25)$) (R.south)};
    \draw [->,dashed,thick] plot [smooth] coordinates {(X.east) ($(H.center) - (0,0.25)$) (Y.west)};
    \draw[->,very thick] (Hd) -- node[pos=0.35,xshift=0.25cm,fill=white](Afg){$ A_{f,g} $} (H) ;
    \node[above right=0.5cm and 0.75cm of Afg.center,dashed,draw](Afgh){\footnotesize $ A_{f,g}h=\frac{\partial h}{\partial x}\begin{bmatrix}f&g\end{bmatrix} $};
    \draw[-](Afg) -- (Afgh.west);
    \draw[->,very thick] (H) -- node[pos=0.6,xshift=0.25cm,fill=white](Mu){$ M_{\overline{\mu}} $} (Hr) ;
    \node[below right=0.25cm and 0.5cm of Mu.center,dashed,draw](MuAfgh){\footnotesize $M_{\overline{\mu}}A_{f,g}h=\frac{\partial h}{\partial x}\begin{bmatrix}f&g\end{bmatrix}\begin{bmatrix}1\\\mu\end{bmatrix}$};
    \draw[-](Mu) -- (MuAfgh.west);
\end{tikzpicture}
    \caption{A schematic diagram of the construction presented in Section \ref{sec:OperatorRepresentation}. The RKHSs are represented by filled circles. The squares at the endpoints of the dashed arrows passing through the circles indicate the domains and co-domains of the functions contained in the RKHSs. The thick arrows between the RKHSs indicate operators.}
    \label{fig:operator-diagram}
\end{figure}
\subsection{Control Occupation Kernels\label{subsec:controlOccupationKernels}}
To facilitate the computation of a finite-rank representation of $A_{F_{\mu}} = M_{\overline{\mu}} A_{f,g} $, and subsequently, the approximate eigenfunctions required for DMD, trajectories of controlled dynamical systems are embedded within vvRKHSs using the so-called \emph{control occupation kernels}. Control occupation kernels arise from a generalization of the idea of \emph{occupation kernels} introduced in \cite{SCC.Rosenfeld.Russo.eatoappear} as follows.
\begin{definition}[\!\!\cite{SCC.Rosenfeld.Russo.eatoappear}]\label{def:occKer}
 Given a continuous function $\gamma:[0,T] \to X$, an RKHS of continuous functions  $\tilde{H}$, and the bounded functional $\mathcal{T}:\tilde{H}\to\mathbb{C}$ defined as  $\mathcal{T}h = \int_0^T h(\gamma(\tau)) d\tau$  for all $h\in \tilde{H}$ the unique function $\Gamma_{\gamma}\in\tilde{H}$  that satisfies $\mathcal{T}h=\langle h, \Gamma_{\gamma}\rangle_{\tilde H}$ for all $h\in\tilde{H}$ is called the \emph{occupation kernel corresponding to $\gamma$ in $\tilde H$}.
\end{definition}

Note that the existence of a unique occupation kernel follows from the Riesz representation theorem. An extension of the definition above to systems with control results in the following notion of a control occupation kernel.
\begin{definition}[\!\!\cite{SCC.Rosenfeld.Kamalapurkar2021}]\label{def:control_occker}
Given a bounded measurable function $u: [0,T] \to \mathbb{R}^m$, a continuous function $\gamma:[0,T] \to X$, and the bounded functional $\mathcal{T} : H \to \mathbb{C}$, defined as
    \begin{equation*} 
        \mathcal{T}h \coloneqq \int_{0}^{T} h(\gamma(t))\begin{bmatrix}1 \\ u(t) \end{bmatrix} \mathrm{d}t,\quad\forall h\in H,
    \end{equation*}
the unique function $\Gamma_{\gamma,u} \in H$  that satisfies $\mathcal{T}h = \langle h, \Gamma_{\gamma,u}\rangle_H$ for all $h \in H$ is called \emph{the control occupation kernel corresponding to $u$ and $\gamma$ in $H$.}
\end{definition}
Control occupation kernels can be expressed in terms of the reproducing kernels of $H$ to facilitate computation.
\begin{proposition}[\!\!\cite{SCC.Rosenfeld.Kamalapurkar2021}]\label{prop:occ_ker_representation}
    The control occupation kernel $\Gamma_{\gamma,u} \in H$, corresponding to $u$ and $\gamma$, can be expressed as
    \begin{equation}
        \Gamma_{\gamma,u}(x) = \int_{0}^{T} \left[\begin{bmatrix} 1 & u(t)^{\top} \end{bmatrix} K_{\gamma\left(t\right)}\right]\left(x\right) \mathrm{d}t,\label{eq:occupationkernel-eval}
    \end{equation}
	and the norm of $\Gamma_{\gamma,u}$ is given as \begin{equation*}
		\| \Gamma_{\gamma,u} \|_H^2 = \int_{0}^{T} \int_{0}^{T} \begin{bmatrix} 1 & u(t)^{\top} \end{bmatrix} K(\gamma(\tau),\gamma(t)) \begin{bmatrix} 1 \\ u(\tau) \end{bmatrix} \mathrm{d}t d\tau.
	\end{equation*}
\end{proposition}

\begin{proof}
	For $x \in X$ and $v \in \mathbb{C}^{1 \times (m+1)}$, 
	\begin{gather}
	    \langle \Gamma_{\gamma,u}(x),v \rangle_{\mathbb{C}^{1\times(m+1)}} = \langle \Gamma_{\gamma,u}, v K_{x} \rangle_{H}
	    = \int_{0}^{T} \left[v K_{x}\right](\gamma(t)) \begin{bmatrix} 1 \\ u(t) \end{bmatrix} \mathrm{d}t
	    =\int_{0}^{T} \left\langle \left[v K_{x}\right](\gamma(t)), \begin{bmatrix} 1 & u(t)^\top \end{bmatrix}\right\rangle_{\mathbb{C}^{1\times(m+1)}} \mathrm{d}t\nonumber\\
	    =\int_{0}^{T} \left\langle v K_{x}, \begin{bmatrix} 1 & u(t)^\top \end{bmatrix} K_{\gamma(t)}\right\rangle_{H} \mathrm{d}t
	    =\int_{0}^{T} \left\langle \left[\begin{bmatrix} 1 & u(t)^\top \end{bmatrix} K_{\gamma(t)}\right](x),v \right\rangle_{\mathbb{C}^{1\times(m+1)}} \mathrm{d}t\nonumber\\
	    = \left\langle \int_{0}^{T}\left[\begin{bmatrix} 1 & u(t)^\top \end{bmatrix} K_{\gamma(t)}\right](x)\mathrm{d}t,v \right\rangle_{\mathbb{C}^{1\times(m+1)}} .\label{eq:occupationkernel-eval-v}
	\end{gather}
	As \eqref{eq:occupationkernel-eval-v} holds for all $v\in \mathbb{C}^{1\times(m+1)}$, \eqref{eq:occupationkernel-eval} follows. The expression for the norm of $\Gamma_{\gamma,u}$ follows from $\| \Gamma_{\gamma,u} \|^2_H = \left\langle  \Gamma_{\gamma,u}, \Gamma_{\gamma,u}  \right\rangle_H$ and the defining properties of $\Gamma_{\gamma,u}$. 
\end{proof}

\subsection{Control Liouville Operators and Control Occupation Kernels}
There is a direct connection between the adjoints of \emph{densely defined} control Liouville operators and control occupation kernels that correspond to \emph{admissible} (see Definition \ref{def:admissible}) control signals, $u$, and their corresponding controlled trajectories, $\gamma_u$, that satisfy \eqref{eq:control-affine}. To illustrate the connection, the construction of adjoints of densely defined operators is revisited in the following.
\begin{definition}
    The \emph{domain of the adjoint of $A: \mathcal{D}(A) \to H$}, with $\mathcal{D}(A) \subseteq \tilde{H}_d$, is defined as
    \begin{equation*} 
        \mathcal{D}\left(A^{*}\right)\coloneqq\left\{h\in H\mid \phi\mapsto\left\langle A \phi,h\right\rangle_{H}\text{ is bounded on }\mathcal{D}\left(A\right)\right\} .
    \end{equation*}
    If $\mathcal{D}(A)$ is dense in $\tilde{H}_d$, then the functionals $\phi\mapsto\left\langle A \phi,h\right\rangle_{H}$ may be extended uniquely to functionals that are bounded over all of $\tilde{H}_{d}$. As a result, for each $ h \in\mathcal{D}\left(A^{*}\right) $, the Riesz representation theorem guarantees the existence of a unique function $A^{*}h\in\tilde{H}_d$ such that $\left\langle A^{*}h,\phi\right\rangle_{\tilde{H}} = \left\langle A\phi,h\right\rangle_{H} $ for all $\phi\in\mathcal{D}(A)$. The operator $h\mapsto A^{*}h$ is defined as \emph{the adjoint of $A$}.
\end{definition}

The following proposition formalizes the relationship between control occupation kernels and control Liouville operators for trajectories of the system under \emph{admissible} control signals.
\begin{definition}\label{def:admissible}
    A bounded, measurable control signal $u:[0,T]\to \mathbb{R}^m$ is called \emph{admissible for the initial value problem \eqref{eq:openloop_dynamics} over the time interval $[0,T]$ and the domain $X$}, if the corresponding Carath\'{e}odory solution $\gamma_u:[0,T]\to\mathbb{R}^n$ is contained within $X$.
\end{definition}
\begin{proposition}[\!\!\cite{SCC.Rosenfeld.Kamalapurkar2021}]\label{prop:adjoint_kernel_difference}
    If $f$ and $g$ correspond to a densely defined control Liouville operator, $A_{f,g} : \mathcal{D}(A_{f,g}) \to H$, with $\mathcal{D}(A_{f,g})\subset \tilde{H}_d$ and $u$ is an admissible control signal for the initial value problem \eqref{eq:openloop_dynamics} over the time interval $[0,T]$, with a corresponding controlled trajectory $\gamma_u$, then $\Gamma_{\gamma_u,u} \in \mathcal{D}(A_{f,g}^*)$ and
    \begin{equation}\label{eq:adjoint-relationship}
        A_{f,g}^* \Gamma_{\gamma_u,u} = \tilde K_d(\cdot, \gamma_u(T)) - \tilde K_d(\cdot,\gamma_u(0)).
    \end{equation}
\end{proposition}
\begin{proof}
    To demonstrate that $\Gamma_{\gamma_u,u}$ is in $\mathcal{D}(A_{f,g}^*)$ it must be shown that the mapping $h \mapsto \langle A_{f,g} h, \Gamma_{\gamma_u,u}\rangle_{H}$ is a bounded functional. Note that
    \begin{multline}
    \langle A_{f,g} h, \Gamma_{\gamma_u,u}\rangle_{H}
    = \int_{0}^{T} \frac{\partial}{\partial x} \Re(h(\gamma_u(t))) \begin{bmatrix} f(\gamma_u(t)) & g(\gamma_u(t)) \end{bmatrix} \begin{bmatrix} 1\\ u(t) \end{bmatrix} \mathrm{d}t\\
    +\int_{0}^{T} \frac{\partial}{\partial x} \Im(h(\gamma_u(t))) \begin{bmatrix} f(\gamma_u(t)) & g(\gamma_u(t)) \end{bmatrix} \begin{bmatrix} 1\\ u(t) \end{bmatrix} \mathrm{d}t
    = \int_{0}^{T} \frac{\mathrm{d}}{\mathrm{d}t} h(\gamma_u(t)) \mathrm{d}t = h(\gamma_u(T)) - h(\gamma_u(0))\\
    = \left\langle h, \tilde K_d(\cdot,\gamma_{u}(T)) - \tilde K_d(\cdot,\gamma_{u}(0)) \right\rangle_{\tilde{H}_{d}},\label{eq:adjoint_action_proof}
    \end{multline}
    where $\Re(\cdot)$ and $\Im(\cdot)$ denote the real and imaginary parts of a complex vector, respectively. The functional $h \mapsto \langle A_{f,g} h, \Gamma_{\gamma_u,u}\rangle_{H}$ is thus bounded with norm not exceeding $\| \tilde K_d(\cdot,\gamma_{u}(T)) - \tilde K_d(\cdot,\gamma_{u}(0))\|_{\tilde{H}_d}$. By the definition of the adjoint,  \eqref{eq:adjoint_action_proof} implies \eqref{eq:adjoint-relationship}.
\end{proof}
\subsection{Properties of Multiplication Operators\label{subsec:MultiplicationOperators}}
In this section, multiplication operators that map from vvRKHSs to scalar-valued RKHSs are studied. Many of the propositions in this section have been established for scalar-valued RKHSs (cf. \cite{SCC.Rosenfeld2015a,SCC.Rosenfeld2015,SCC.Szafraniec2000}), and are proved using similar methods.

The following proposition investigates the interaction between adjoints of multiplication operators and reproducing kernels of scalar-valued RKHSs.
\begin{proposition}[\!\!\cite{SCC.Rosenfeld.Kamalapurkar2021}]\label{prop:mult-adjoint-kernel}
    If $\nu:X\to\mathcal{Y}$ corresponds to a densely defined multiplication operator $M_{\nu}:\mathcal{D}\left(M_{\nu}\right)\to \tilde{H}_r$ with $\mathcal{D}\left(M_{\nu}\right)\subset H$, then for each $x\in X$, $\tilde K_r(\cdot,x)$ is in the domain of $M_{\nu}^*$ and
    \begin{equation}
        M_{\nu}^* \tilde K_r(\cdot,x) = K_{x,\nu(x)}.\label{eq:mult-adjoint-relation}
    \end{equation}
\end{proposition}

\begin{proof}
    Let $h \in \mathcal{D}(M_{\nu})$, then 
    \begin{equation*}\
    \langle M_{\nu} h, \tilde K_r(\cdot,x) \rangle_{\tilde{H}_{r}} = \langle h(x), \nu(x) \rangle_{\mathcal{Y}} = \langle h, K_{x,\nu(x)} \rangle_H.
  \end{equation*} 
  Hence, the mapping $h \mapsto \langle M_{\nu} h, \tilde K_r(\cdot,x) \rangle_{\tilde{H}_{r}}$ is a bounded functional with norm bounded by $\| K_{x,\nu(x)} \|_{H}$, and as such, $\tilde K_r(\cdot,x)$ is in the domain of $M_{\nu}^{*}$. Moreover, the equation \begin{equation*} \langle M_{\nu} h,\tilde  K_r(\cdot,x) \rangle_{\tilde{H}_{r}} = \langle h, K_{x,\nu(x)} \rangle_H, \end{equation*} along with the definition of the adjoint, establishes \eqref{eq:mult-adjoint-relation}. 
\end{proof}
Proposition \ref{prop:mult-adjoint-kernel}, along with the density of the kernels $\tilde{K}(\cdot,x)$ in $\tilde{H}$ implies that the adjoint of the multiplication operator is also densely defined.
\begin{proposition}[\!\!\cite{SCC.Rosenfeld.Kamalapurkar2021}]
    Multiplication operators are closed operators.
\end{proposition}
\begin{proof}
    Suppose that $\{ h_n \}_{n=1}^\infty \subset \mathcal{D}(M_{\nu})$, $h_{n} \to h \in H$, and $M_{\nu} h_n \to W \in \tilde{H}_{r}.$ To show that $M_{\nu}$ is a closed operator, it must be shown that $W(x) = \langle h(x), \nu(x) \rangle_{\mathcal{Y}}$ for all $x \in X$, and thus $h \in \mathcal{D}(M_{\nu})$ by definition and $M_{\nu} h = W$. Let $v \in \mathcal{Y}$ and $x \in X$, then
    \begin{gather*}
        W(x) = \lim_{n \to \infty} \langle M_{\nu} h_n, \tilde K_r(\cdot,x) \rangle_{\tilde{H}_{r}}
        = \lim_{n\to\infty} \langle h_n, K_{x,\nu(x)} \rangle_{H} = \langle h, K_{x,\nu(x)} \rangle_{H} = \langle h(x), \nu(x) \rangle_{\mathcal{Y}},
    \end{gather*}
    where the first equality follows since norm convergence in $\tilde{H}_{r}$ implies pointwise convergence and the third inequality follows from continuity of the inner product on $H$. 
\end{proof}
 
The following proposition demonstrates how multiplication operators $M_{\overline{\mu}}$, with symbols $\overline{\mu}$ given as $\overline{\mu}(x) = \begin{bmatrix}1 & \mu(x)^{\top}\end{bmatrix},\forall x\in\mathbb{R}^n $, connect occupation kernels $\Gamma_{\gamma}$ with feedback control occupation kernels $\Gamma_{\gamma,\mu\circ\gamma}$.
\begin{proposition}\label{prop:mult_adjoint_feedback}
    If $\mu : X \to \mathbb{R}^m$ is a continuous function, $\overline{\mu} : X\to\mathcal{Y}$ is defined as $\overline{\mu}(x) := \begin{bmatrix}1 & \mu(x)^{\top}\end{bmatrix},\forall x\in X $, the corresponding multiplication operator $M_{\overline{\mu}}:\mathcal{D}\left(M_{\overline{\mu}}\right)\to \tilde{H}_r$, with $\mathcal{D}\left(M_{\overline{\mu}}\right)\subset H$, is densely defined, $\Gamma_{\gamma}\in\tilde{H}_{r}$ is the occupation kernel corresponding to a continuous function $\gamma: [0,T] \to X$ in $\tilde{H}_r$, and $ \left\Vert\Gamma_{\gamma,\mu\circ\gamma}\right\Vert_{H}$ is finite, then $\Gamma_{\gamma}$ is in the domain of $M_{\overline{\mu}}^{*}$ and
    \begin{equation}\label{eq:adjoint-occ-Cocc}
        M_{\overline{\mu}}^* \Gamma_{\gamma} = \Gamma_{\gamma,\mu\circ\gamma}.
    \end{equation}
\end{proposition}
\begin{proof}
    Let $h\in\mathcal{D}\left(M_{\overline{\mu}}\right)$. Using definitions \ref{def:occKer} and \ref{def:control_occker}, it can be concluded that
    \begin{equation}
        \left\langle M_{\overline{\mu}}h,\Gamma_\gamma\right\rangle_{\tilde{H}} = \int_{0}^{T} \left\langle h(\gamma(t)),\overline{\mu}(\gamma(t)) \right\rangle_{\mathcal{Y}} \mathrm{d}t
        = \int_{0}^{T} \left(\Re\left(h\left(\gamma\left(t\right)\right)\right)+\Im\left(h\left(\gamma\left(t\right)\right)\right)\right)\begin{bmatrix}1\\\mu(\gamma(t))\end{bmatrix} \mathrm{d}t
        = \left\langle h,\Gamma_{\gamma,\mu\circ\gamma}\right\rangle_{H},\label{eq:adjoint-proof}
    \end{equation}
    Since the norm of the functional $h\mapsto\left\langle M_{\overline{\mu}}h,\Gamma_\gamma\right\rangle_{\tilde{H}_r}$ is bounded, by $\left\Vert \Gamma_{\gamma,\mu\circ\gamma} \right\Vert_{H}$, which in turn, is finite by assumption, it can be concluded that $\Gamma_\gamma \in \mathcal{D}(M^*_{\overline{\mu}})$. As a result, \eqref{eq:adjoint-proof} implies \eqref{eq:adjoint-occ-Cocc} and the proof of the proposition is complete.
\end{proof}
\begin{remark}
    If the reproducing kernel $K$ for the vvRKHS is derived from the reproducing kernel $\tilde{K}$ of an RKHS via multiplication by a positive definite matrix, then finiteness of $ \left\Vert\Gamma_{\gamma,\mu\circ\gamma}\right\Vert_{H} $  follows from Proposition \ref{prop:occ_ker_representation} and continuity of $\mu$, $\gamma$, and $\tilde{K}$.
\end{remark}
\subsection{Compact and Densely Defined Operators for DMD}\label{subsec:compactDenselyDefinedOperatorsDMD}
As noted in Section \ref{sec:OperatorDMD}, DMD relies on computation of eigenfunctions of a finite-rank representation of an operator that represents the dynamical system. The eigenfunctions of the finite-rank representations can be shown to converge to the eigenfunctions of the true operator if the finite-rank representations themselves converge to the true operator in the norm topology and the true operator is compact \cite{SCC.Korda.Mezic2018a}. Koopman operators, Koopman generators, and Liouville operators are typically not compact if their domains and co-domains are viewed as subsets of the same RKHS \cite{SCC.Korda.Mezic2018a,SCC.Rosenfeld.Kamalapurkar2023a,SCC.Gonzalez.Abudia.easubmitted} (see Remark \ref{rem:compact_perspective}).

As noted in \cite{SCC.Rosenfeld.Kamalapurkar2023a}, Liouville operators corresponding to a large class of dynamical systems are compact provided the domain and the range RKHSs are selected appropriately. However, since the domain and the range RKHSs need to be different, the resulting operators do not admit eigenfunctions. In Section \ref{sec:SVD-DMD}, it is shown that when the domain and the range RKHSs are different, an SVD-based approach can be used to estimate the system dynamics. The SVD-based approach relies on compactness of the total derivative operator and generates sequences of singular values and singular functions that converge to the true singular values and singular functions. As shown in Section \ref{sec:SVD-DMD}, compact total derivative operators result from bounded multiplication operators and compact control Liouville operators, both of which exist for a large class of dynamical systems and feedback laws.

In Section \ref{sec:eig-DMD}, an eigendecomposition-based DMD approach is developed that lacks convergence guarantees but generates useful heuristic approximations of the eigenfunctions under the weaker assumption that the total derivative operator is densely defined. As shown in Section \ref{sec:eig-DMD}, a densely defined total derivative operator results from a densely defined multiplication operator whose range is a subset of the domain of a densely defined control Liouville operator. As discussed in Section \ref{sec:eig-DMD}, such multiplication operators and control Liouville operators also exist for a large class of dynamical systems and feedback laws.

In the following, for an operator $A$ and finite collections of functions $d$ and $r$, in the domain and the range of the operator, respectively, the notation $A\vert_d$ is used to denote the operator $A$ restricted to the set $\vspan d$, and the notation $[A]_{d}^{r}$ is used to denote a matrix representation of the finite-rank operator $ P_r A\vert_d $, where $P_r$ denotes the projection operator onto $\vspan r$.

\section{A Singular Value Decomposition Approach to DMD\label{sec:SVD-DMD}}
With careful selection of the domain and range RKHSs, the total derivative operator $M_{\overline \mu} A_{f,g}$ can be made to be compact.  While the provided framework includes a large class of dynamical systems, a complete characterization of RKHSs and symbols that yield compact differential operators and bounded multiplication operators is out of the scope of this paper.

\subsection{Existence of Bounded Multiplication Operators and Compact Differential Operators}\label{subsec:existence-compactness}
The discussion in this section closely follows \cite{SCC.Rosenfeld.Kamalapurkar2023a}, where a similar result is obtained for systems without control. Consider the exponential dot product kernel with parameter $\tilde{\rho}$, defined as $\tilde{K}_{\tilde{\rho}}(x,y) = \exp\left(\frac{x^{\top}y}{\tilde{\rho}}\right)$. In the single variable case, the native space\footnote{The \emph{native space} of a symmetric positive semidefinite kernel $K$ is the unique RKHS of which $K$ is the reproducing kernel. Such an RKHS is guaranteed to exist by the Moore–Aronszajn theorem \cite{SCC.Aronszajn1950}.} for this kernel is the restriction of the Bargmann-Fock space to real numbers, denoted by $F^2_{\tilde{\rho}}\left(\mathbb{R}\right)$. This space consists of the set of functions of the form $h(x) = \sum_{k=0}^\infty a_k x^k$, where the coefficients satisfy $\sum_{k=0}^\infty \left\vert a_k\right\vert^2 \tilde{\rho}^k k! < \infty$, and the norm is given by $ \left\Vert h \right\Vert^{2}_{\tilde{\rho}} = \sum_{k=0}^\infty \left\vert a_k\right\vert^2 \tilde{\rho}^k k! $. Note that the set of polynomials in $x$ is a subset of $F^2_{\tilde{\rho}}\left(\mathbb{R}\right)$. Extension of this definition to the multivariable case yields the space $F^2_{\tilde{\rho}}\left(\mathbb{R}^n\right)$ where the collection of monomials, $x^{\alpha} \frac{\tilde{\rho}^{|\alpha|}}{\sqrt{\alpha!}}$, with multi-indices $\alpha \in \mathbb{N}^n$ forms an orthonormal basis\footnote{For $\alpha \in \mathbb{N}^n$, $\alpha! = \prod_{i=1}^n \alpha_i !$, $|\alpha| = \sum_{i=1}^n \alpha_i$, and $x^\alpha = \prod_{i=1}^n x_i^{\alpha_i}$.}.
In this setting, provided  $\tilde{\rho}_2 < \tilde{\rho}_1$, differential operators from $F^2_{\tilde{\rho}_1}(\mathbb{R}^n)$ to $F^2_{\tilde{\rho}_2}(\mathbb{R}^n)$ can be shown to be compact.
\begin{proposition}\label{prop:differential_compact}
    If $\tilde{\rho}_2 < \tilde{\rho}_1$, then the differential operators $\frac{\partial}{\partial x_i}  : F^2_{\tilde{\rho}_1}(\mathbb{R}^n) \to F^2_{\tilde{\rho}_2}(\mathbb{R}^n)$, are compact for $i=1,\ldots,n$. 
\end{proposition}
\begin{proof}
 To facilitate the clarity of exposition, the proof is written for functions of a single variable. Extension to functions of several variables using multi-indices is conceptually straightforward. Let $h\in F^2_{\tilde{\rho}_1}(\mathbb{R}) $ be given by $h(x) = \sum_{k=0}^\infty a_k x^k$, with $\frac{\partial h}{\partial x} = \sum_{k=0}^\infty b_k x^k $, where $b_k = (k+1) a_{k+1}$. The norm of the derivative in $F^2_{\tilde{\rho}_2}(\mathbb{R})$ is given by
\begin{equation*}
     \left\Vert\frac{\partial h}{\partial x}\right\Vert^{2}_{\tilde{\rho}_2} = \sum_{k=0}^\infty \left\vert b_k\right\vert^2 \tilde{\rho}_2^k k! 
     = \sum_{k=0}^\infty \left\vert a_{k+1}\right\vert^2 \frac{k+1}{\tilde{\rho}_1}\left(\frac{\tilde{\rho}_2}{\tilde{\rho}_1}\right)^k\tilde{\rho}_1^{(k+1)} (k+1)!
     = \sum_{k=0}^\infty \left\vert a_{k}\right\vert^2 \frac{k}{\tilde{\rho}_1}\left(\frac{\tilde{\rho}_2}{\tilde{\rho}_1}\right)^{(k-1)}\tilde{\rho}_1^{k} k!
\end{equation*}
If $\tilde{\rho}_2 < \tilde{\rho}_1$ then there exists a constant $C<\infty$ such that  $ \frac{k}{\tilde{\rho}_1}\left(\frac{\tilde{\rho}_2}{\tilde{\rho}_1}\right)^{(k-1)} \leq C$ for all $k\in\mathbb{N}$. As a result, $\left\Vert\frac{\partial h}{\partial x}\right\Vert^{2}_{\tilde{\rho}_2} \leq C \left\Vert h\right\Vert^{2}_{\tilde{\rho}_1}$, which establishes boundedness of the differential operator $\frac{\partial}{\partial x}  : F^2_{\tilde{\rho}_1}(\mathbb{R}) \to F^2_{\tilde{\rho}_2}(\mathbb{R})$. 

To prove compactness, we construct a sequence of finite-rank operators that converge, in norm, to $\frac{\partial}{\partial x}$.  Let $\alpha_M := \{ 1, x, \ldots, x^M\}$ be the first $M$ monomials in $x$, and let $P_{\alpha_M}$ be the projection onto the span of these monomials. Consider the sequence $\{P_{\alpha_M}\frac{\partial}{\partial x}\}$ of finite-rank operators. Let $h\in F^2_{\tilde{\rho}_1}(\mathbb{R}) $ be given by $h(x) = \sum_{k=0}^\infty a_k x^k$. Then,
\begin{multline*}
    \left\Vert P_{\alpha_M}\frac{\partial h}{\partial x} - \frac{\partial h}{\partial x} \right\Vert^2_{\tilde{\rho}_2} = \!\!\sum_{k = M+1}^{\infty} (k+1)^2 \left\vert a_{k+1}\right\vert^2 \tilde{\rho}_2^{(k+1)} (k+1)!
   = \sum_{k = M+1}^{\infty} (k+1)^2\left(\frac{\tilde{\rho}_2}{\tilde{\rho}_1}\right)^{(k+1)} \sum_{k = M+1}^{\infty} \left\vert a_{k+1}\right\vert^2\tilde{\rho}_1^{(k+1)} (k+1)!\\
   \leq \sum_{k = M+1}^{\infty} (k+1)^2\left(\frac{\tilde{\rho}_2}{\tilde{\rho}_1}\right)^{(k+1)} \left\Vert h\right\Vert^{2}_{\tilde{\rho}_1}.
\end{multline*}
If $\tilde{\rho}_2 < \tilde{\rho}_1$ then {\thickmuskip=0mu\thinmuskip=0mu\medmuskip=0mu$ \lim_{M\to\infty} \sum_{k = M+1}^{\infty} (k+1)^2 \left(\frac{\tilde{\rho}_2}{\tilde{\rho}_1}\right)^{(k+1)} = 0 $}, and as a result, $\lim_{M\to\infty} \left\Vert P_{\alpha_M}\frac{\partial h}{\partial x} - \frac{\partial h}{\partial x} \right\Vert^2_{\tilde{\rho}_2} = 0$. Therefore, the operator norm 
 \[
    \left\| P_{\alpha_M} \frac{\partial}{\partial x} - \frac{\partial}{\partial x} \right\|_{F_{\tilde{\rho}_1}^2(\mathbb{R})}^{F_{\tilde{\rho}_2}^2(\mathbb{R})} := \sup_{h \in F_{\tilde{\rho}_1}^2(\mathbb{R})} \frac{\| P_{\alpha_M} \frac{\partial h}{\partial x} - \frac{\partial h}{\partial x}\|_{\tilde{\rho}_2}}{\|h\|_{\tilde{\rho}_1}}
\]
converges to zero as $M \to\infty$, which establishes compactness of  $\frac{\partial}{\partial x}  : F^2_{\tilde{\rho}_1}(\mathbb{R}) \to F^2_{\tilde{\rho}_2}(\mathbb{R})$. 
\end{proof}
\begin{remark}{\label{rem:compact_perspective}}
    Note that if $\tilde{\rho}_2 < \tilde{\rho}_1$ then $F^2_{\tilde{\rho}_1}(\mathbb{R}^n) \subset F^2_{\tilde{\rho}_2}(\mathbb{R}^n)$ \cite[Proposition 5.1]{SCC.Rosenfeld.Kamalapurkar2023a}. In this case, one can view the differential operators as maps from $F^2_{\tilde{\rho}_2}(\mathbb{R}^n)$ to itself. However, when viewed as such, the differential operators may not be compact.
\end{remark}

As shown in \cite{SCC.Rosenfeld.Kamalapurkar2023a}, multiplication operators can be shown to be bounded provided their symbols are polynomial.
\begin{proposition}\label{prop:multiplication_bounded}
    If $\tilde{\rho}_2 < \tilde{\rho}_1$, then for any polynomial function $p:\mathbb{R}^n \to \mathbb{R}$
    , the multiplication operator $M_{p} : F^2_{\tilde{\rho}_1}(\mathbb{R}^n) \to F^2_{\tilde{\rho}_2}(\mathbb{R}^n)$, defined as $\left[M_p h\right](x) = p(x)h(x)$, is bounded.
\end{proposition}
\begin{proof}
    See \cite[Lemma 3.2]{SCC.Rosenfeld.Kamalapurkar2023a}.
\end{proof}
Proposition \ref{prop:multiplication_bounded} trivially extends to vvRKHSs defined using diagonal reproducing kernels.
\begin{proposition}\label{prop:multiplication_bounded_vvRKHS}
    Let $F^2_{\rho}(\mathbb{R}^n)$ denote the native (row) vvRKHS of a diagonal reproducing kernel defined as $K(x,y) := \diag \left(\begin{bmatrix} \tilde{K}_{\rho_1}(x,y), & \ldots, & \tilde{K}_{\rho_{m+1}}(x,y) \end{bmatrix} \right) $. If $\rho_i < \tilde{\rho}$ for $i=1,\ldots,m+1$, then given any set of polynomials $p_i$, $i=1,\ldots,m+1$, the multiplication operator $M_{p_1,\ldots,p_{m+1}} : F^2_{\tilde{\rho}}(\mathbb{R}^n) \to F^2_{\rho}(\mathbb{R}^n)$, defined as $\left[M_{p_1,\ldots,p_{m+1}} h\right](x) = h(x)\begin{bmatrix}
        p_1(x),&\ldots,&p_{m+1}(x)
    \end{bmatrix}$, is bounded. On the other hand, if $ \tilde{\rho} < \rho_i $ for $i=1,\ldots,m+1$, then for any component-wise polynomial function $\mu = \begin{bmatrix} \mu_1,\ldots,\mu_m\end{bmatrix}^\top:\mathbb{R}^n \to \mathbb{R}^m$, the multiplication operator $M_{\overline{\mu}}:F^2_{\rho}(\mathbb{R}^n)\to F^2_{\tilde{\rho}}(\mathbb{R}^n)$, defined as $\left[M_{\overline{\mu}} h\right](x) = h(x) \begin{bmatrix} 1,&\mu_1(x),&\ldots,&\mu_m(x)\end{bmatrix}^{\top} $, is bounded.
\end{proposition}
\begin{proof}
    Follows from arguments similar to Lemma 3.2 from \cite{SCC.Rosenfeld.Kamalapurkar2023a}.
\end{proof}

Since Koopman operators are generally unbounded for any nonlinear system \cite{SCC.Gonzalez.Abudia.easubmitted}, the above propositions make a strong case for spectral analysis of continuous-time systems in the Liouville operator (or Koopman generator) framework as opposed to discretization and subsequent application of the Koopman operator framework.

\subsection{Finite-rank Representation of the Closed Loop Total Derivative Operator}
Since the dynamic modes may only be extracted from the composition of $M_{\overline{\mu}}$ with $A_{f,g}$, an explicit finite-rank representation of $A_{f,g}$ and $M_{\overline{\mu}}$ is needed to determine the dynamic modes of the resultant system. In the following, finite collections of linearly independent vectors, $d^M$, $\varpi^M$, $\beta^M$, and $r^M$ are selected to establish the needed finite-rank representation. Since the adjoint of $ A_{f,g} $ maps control occupation kernels to kernel differences (Proposition \ref{prop:adjoint_kernel_difference}), the span of the collection of kernel differences
\begin{equation}
    d^M = \left\{K_d(\cdot,\gamma_{u_i}(T_i)) - K_d(\cdot,\gamma_{u_i}(0))\right\}_{i=1}^M\subset \tilde{H}_{d}
\end{equation}
is selected to be the domain of  $ A_{f,g} $. The corresponding Gram matrix is denoted by $G_{d^M} = \left(\left\langle d_i,d_j\right\rangle_{\tilde{H}_d}\right)_{i,j=1}^M$. The output of $A_{f,g}$ is projected onto the span of the control occupation kernels
\begin{equation}
    \beta^M = \left\{\Gamma_{\gamma_{u_i},u_i}\right\}_{i=1}^M\subset H
\end{equation}
before application of $M_{\overline{\mu}}$. The corresponding Gram matrix is denoted by  $G_{\beta^M} = \left(\left\langle\beta_i,\beta_j\right\rangle_{H}\right)_{i,j=1}^M$.

\begin{figure}[h]
    \centering
\begin{tikzpicture}[>= latex]
    \node[minimum size=1.25cm, circle, fill=lightgray, circle] (H) {$H$};
    \node[minimum size=1.25cm, circle, fill=lightgray, above left=0cm and 6cm of H.center, anchor=center] (Hd) {$\tilde{H}_d$};
    \node[minimum size=1cm, fill=lightgray, above left=1cm and 3cm of H.center, anchor=center] (sd) {$\vspan{d^M}$};
    \node[minimum size=1cm, fill=lightgray, below=3cm of H.center, anchor=center] (sb) {$\vspan{\beta^M}$};
    \node[minimum size=1.25cm, fill=lightgray, circle, below left=0cm and 6cm of sb.center, anchor=center] (Hr) {$\tilde{H}_r$};
    \node[minimum size=1cm, fill=lightgray, above right=1.5cm and 3cm of Hr.center, anchor=center] (sr) {$\vspan{r^M}$};
    \draw[->,very thick] (sd) -- node[midway, fill=white](Afg){$ A_{f,g} $} (H) ;
    \draw[->,very thick] (H) -- node[midway, fill=white](Pb){$ P_{\beta^M} $} (sb) ;
    \draw[->,very thick] (sb) -- node[midway, fill=white](Mu){$ M_{\overline{\mu}} $} (Hr) ;
    \draw[->,very thick] (Hr) -- node[midway, fill=white](Pr){$ P_{r^M} $} (sr) ;
    \draw[->,very thick] (Hd) -- node[midway, fill=white](Pr){$ P_{d^M} $} (sd) ;
    \draw[->,very thick,dashed] (sd) -- node[midway, fill=white]{$ P_{r^M}M_{\overline{\mu}} P_{\beta^M} A_{f,g} |_{d^M} $} (sr) ;
    \draw[->,thick,dashed] (Hd) -- node[midway, fill=white]{$ M_{\overline{\mu}}A_{f,g}$} (Hr) ;
\end{tikzpicture}
    \caption{A schematic diagram of the finite-rank representation of the total derivative operator.}
    \label{fig:finite-rank-schematic}
\end{figure}

Since the adjoint of $M_{\overline{\mu}}$ maps occupation kernels to control occupation kernels of the form $\Gamma_{\gamma_{u_i},\mu\circ\gamma_{u_i}}$ (Proposition \ref{prop:mult_adjoint_feedback}), the derivation also requires the collection
\begin{equation}
    \varpi^M = \left\{\Gamma_{\gamma_{u_i},\mu\circ\gamma_{u_i}}\right\}_{i=1}^M\subset H
\end{equation}
of feedback control occupation kernels in $H$ corresponding to the trajectories $\gamma_{u_i}$ and control signals $\mu\circ\gamma_{u_i}$.
Finally, the result of $M_{\overline{\mu}}$ is projected onto the span of the occupation kernels
\begin{equation}
    r^M = \left\{\Gamma_{\gamma_{u_i}}\right\}_{i=1}^M\subset \tilde{H}_{r}.
\end{equation}
The corresponding Gram matrix is denoted by $G_{r^M} = \left(\left\langle r_i,r_j\right\rangle_{\tilde{H}_r}\right)_{i,j=1}^M$. 

A rank-$M$ representation of the operator $M_{\overline{\mu}} A_{f,g}$ is then given by  $P_{r^M} M_{\overline{\mu}} P_{\beta^M} A_{f,g}P_{d^M}:\tilde{H}_d\to\vspan r^M$, where $P_{r^M}$, $P_{d^M}$, and $P_{\beta^M}$ denote projection operators onto $\vspan{r^M}$, $\vspan{d^M}$, and $\vspan{\beta^M}$, respectively. The construction is illustrated in Fig. \ref{fig:finite-rank-schematic}.

Under the compactness assumptions and given rich enough data so that the spans of $ \{d_i\}_{i=1}^\infty $, $\{r_i\}_{i=1}^\infty$, and $\{\beta_i\}_{i=1}^\infty $ are dense in $\tilde{H}_{d}$, $\tilde{H}_{r}$, and $H$, respectively, the sequence of finite-rank operators $\{P_{r^M} M_{\overline{\mu}} P_{\beta^M} A_{f,g}P_{d^M}\}_{M=1}^{\infty}$ can be shown to converge, in norm topology, to $M_{\overline{\mu}}A_{f,g}$. To facilitate the proof of convergence, we recall the following result from \cite{SCC.Rosenfeld.Kamalapurkar.ea2022}.
\begin{lemma}\label{lem:technical-lemma}
    Let $H$ and $G$ be RKHSs defined on $X\subset \mathbb{R}^n$ and let $A_N:H\to G$ be a finite-rank operator with rank $N$. If the spans of $\{d_i\}_{i=1}^{\infty}$ and $\{r_i\}_{i=1}^{\infty}$ are dense in $H$  and $G$, respectively, then for all $\epsilon > 0 $, there exists $M(N)\in\mathbb{N}$ such that for all $i\geq M(N)$ and $h\in H$, $\left\Vert A_N h - A_N P_{d^i} h\right\Vert_{G} \leq \epsilon \left\Vert h \right\Vert_H$ and $\left\Vert A_N h - P_{r^i} A_N h\right\Vert_{G} \leq \epsilon \left\Vert h \right\Vert_H$.
\end{lemma}
\begin{proof}
    See the proof of \cite[Theorem 2]{SCC.Rosenfeld.Kamalapurkar.ea2022}.
\end{proof}
The convergence result for Liouville operators on Bargmann-Fock spaces restricted to the set of real numbers follows from the following more general result.
\begin{proposition}\label{prop:convergence}
    If $B:H\to \tilde{H}_r$ is a bounded linear operator, $A:\tilde{H}_d\to H$ is a compact operator, and the spans of $ \{d_i\}_{i=1}^\infty $, $\{r_i\}_{i=1}^\infty$, and $\{\beta_i\}_{i=1}^\infty $ are dense in $\tilde{H}_{d}$, $\tilde{H}_{r}$, and $H$, respectively, then $\lim_{M\to\infty}\left\Vert BA - P_{r^M} B P_{\beta^M} A P_{d^M} \right\Vert_{\tilde{H}_d}^{\tilde{H}_r} = 0$, where $\left\Vert\cdot\right\Vert_{\tilde{H}_d}^{\tilde{H}_r}$ denotes the operator norm of operators from $\tilde{H}_d$ to $\tilde{H}_r$.
\end{proposition}
\begin{proof}
    Let $\{A_N\}_{N=1}^\infty$ be a sequence of rank-$N$ operators converging, in norm, to $A$. For an arbitrary $h\in \tilde{H}_d$, 
\begin{multline*}
    \left\Vert BAh - P_{r^M} B P_{\beta^M} AP_{d^M} h\right\Vert_{\tilde{H}_r} \leq \left\Vert BAh - BA_{N}h  \right\Vert_{\tilde{H}_r} + \left\Vert BA_{N}h  - P_{r^M}BP_{\beta^M}TP_{d^M}h\right\Vert_{\tilde{H}_r}\\
    \leq \left\Vert BAh - BA_{N}h  \right\Vert_{\tilde{H}_r} 
    + \left\Vert BA_{N}h - BA_{N}P_{d^M}h \right\Vert_{\tilde{H}_r}
    +\left\Vert BA_{N}P_{d^M}h - BP_{\beta^M}A_{N}P_{d^M}h \right\Vert_{\tilde{H}_r} \\
    +\left\Vert BP_{\beta^M}A_{N}P_{d^M}h - P_{r^M}BP_{\beta^M}A_{N}P_{d^M}h \right\Vert_{\tilde{H}_r}     +\left\Vert P_{r^M}BP_{\beta^M}A_{N}P_{d^M}h - P_{r^M}BP_{\beta^M} A P_{d^M}h\right\Vert_{\tilde{H}_r}.
\end{multline*}
Assuming that the operator norm of $B$ is $\overline{B}$,
\begin{multline*}
    \left\Vert BAh - P_{r^M} B P_{\beta^M} AP_{d^M} h\right\Vert_{\tilde{H}_r} 
    \leq \overline{B}\left\Vert Ah - A_{N}h  \right\Vert_{H} + 
    \overline{B}\left\Vert A_{N}h - A_{N}P_{d^M}h \right\Vert_{H} 
    +\overline{B} \left\Vert A_{N}P_{d^M}h - P_{\beta^M}A_{N}P_{d^M}h \right\Vert_{H} \\
    +\left\Vert BP_{\beta^M}A_{N}P_{d^M}h - P_{r^M}BP_{\beta^M}A_{N}P_{d^M}h \right\Vert_{\tilde{H}_r} 
    +\overline{B} \left\Vert A_{N}P_{d^M}h - A P_{d^M}h\right\Vert_{H}.
\end{multline*}
Using the fact that $A_N$ and $BP_{\beta^M}A_{N}P_{d^M}$ are finite-rank operators, Lemma \ref{lem:technical-lemma}, can be used to conclude that for all $\epsilon > 0 $, there exists $M(N)\in\mathbb{N}$ such that for all $i\geq M(N)$
\begin{equation*}
    \left\Vert BAh - P_{r^i} B P_{\beta^i} AP_{d^i} h\right\Vert_{\tilde{H}_r} 
    \leq \overline{B}\left\Vert Ah - A_{N}h  \right\Vert_{H}
    + 3\overline{B}\epsilon\left\Vert h \right\Vert_{\tilde{H}_d}
    +\overline{B} \left\Vert A_{N}P_{d^i}h - A P_{d^i}h\right\Vert_{H}.
\end{equation*}
Since $A_{N}$ converges to $A$ in norm, given $\epsilon > 0$, there exists $N\in \mathbb{N}$ such that for all $j\geq N$, and  $g\in \tilde{H}_d$ $\left\Vert Ag - A_{j}g  \right\Vert_{H} \leq \epsilon \left\Vert g  \right\Vert_{\tilde{H}_d} $. Thus, for all  $j\geq N$ and  $i\geq M(j)$, $ \left\Vert BAh - P_{r^i} B P_{\beta^i} AP_{d^i} h\right\Vert_{\tilde{H}_r}
    \leq 5\overline{B}\epsilon\left\Vert h \right\Vert_{\tilde{H}_d}$.
\end{proof}
The convergence result for control Liouville operators on Bargmann-Fock spaces restricted to the set of real numbers can then be stated as follows.
\begin{theorem}\label{thm:norm-convergence}
    Let $\rho_d\in\mathbb{R}$, $\varrho_d\in\mathbb{R}$, $\rho_r\in\mathbb{R}$, and $\rho=\begin{bmatrix}
        \rho_1&\ldots&\rho_{m+1}
    \end{bmatrix}^{\top}\in\mathbb{R}^{m+1}$ be parameters such that  $ \rho_r < \rho_i $, $ \rho_i < \varrho_d $, and $\varrho_d < \rho_d $ for $i = 1,\ldots,m+1$. Let $\tilde{H}_d = F^2_{\tilde{\rho}_d}(\mathbb{R}^n)$, $\tilde{G}_d = F^2_{\tilde{\varrho}_d}(\mathbb{R}^n)$, $\tilde{H}_r = F^2_{\tilde{\rho}_r}(\mathbb{R}^n)$, and $H = F^2_{\rho}(\mathbb{R}^n)$. If $f$, $g$, and $\mu$ are component-wise polynomial, and if the spans of the collections $ \{d_i\}_{i=1}^\infty $, $\{r_i\}_{i=1}^\infty$, and $\{\beta_i\}_{i=1}^\infty $ are dense in $\tilde{H}_{d}$, $\tilde{H}_{r}$, and $H$, respectively, then $\lim_{M\to\infty}\left\Vert M_{\overline{\mu}}A_{f,g} - P_{r^M} M_{\overline{\mu}} P_{\beta^M} A_{f,g} P_{d^M} \right\Vert_{\tilde{H}_d}^{\tilde{H}_r} = 0$.
\end{theorem}
\begin{proof}
    Propositions \ref{prop:differential_compact}, \ref{prop:multiplication_bounded}, and \ref{prop:multiplication_bounded_vvRKHS} imply that $M_{\overline{\mu}}$ is bounded and $A_{f,g}$ is compact. Since multiplication operators are linear by definition, the theorem follows from Proposition \ref{prop:convergence}. 
\end{proof}

\subsection{Matrix Representation of the Finite-rank Operator}
 To formulate a matrix representation of the finite-rank operator $P_{r^M} M_{\overline{\mu}} P_{\beta^M} A_{f,g}P_{d^M} $, the operator is restricted to $\vspan d^M$ to yield the operator $P_{r^M} M_{\overline{\mu}} P_{\beta^M} A_{f,g}|_{d^M}:\vspan{d^M}\to\vspan{r^M}$. For brevity of exposition, the superscript $M$ is suppressed hereafter and $d$, $\beta$, $\varpi$, and $r$ are interpreted as $M-$dimensional vectors. 
\begin{proposition}
If $ h = \delta^\top d\in \vspan{d}$ is a function with coefficients $d\in\mathbb{R}^M$ and if $g = P_r M_{\overline{\mu}} P_\beta A_{f,g} h$, then $g = a^{\top}r$, where $a = G_r^{+} I G_\beta^{+} G_d \delta$, $I := \left(\left\langle \varpi_i,\beta_j \right\rangle_{H}\right)_{i,j=1}^M$, and $(\cdot)^{+}$ denotes the Moore-Penrose pseudoinverse.
\end{proposition}
\begin{proof}
    Propositions \ref{prop:adjoint_kernel_difference} and \ref{prop:mult_adjoint_feedback} imply that that for all $j=1,\cdots,M$, $ A_{f,g}^* \beta_j = d_j$, and $M_{\overline{\mu}}^* r_j = \varpi_j$, respectively. Note that since $g$  is a projection of $M_{\overline{\mu}} P_\beta A_{f,g} h$ onto $\vspan r$, $g = a^{\top}r$ for any $a$ that solves
    \begin{equation} \label{eq:alpha_projection}
        G_r a = \begin{bmatrix}
        \left\langle M_{\overline{\mu}} P_\beta A_{f,g}h,r_1\right\rangle_{\tilde{H}_{r}}\\\vdots\\\left\langle M_{\overline{\mu}} P_\beta A_{f,g}h,r_M\right\rangle_{\tilde{H}_{r}}
        \end{bmatrix} = \begin{bmatrix}
            \left\langle A_{f,g}h,P_\beta M_{\overline{\mu}}^*r_1\right\rangle_{H}\\\vdots\\\left\langle  A_{f,g},P_\beta M_{\overline{\mu}}^*r_M\right\rangle_{H}
        \end{bmatrix}.
    \end{equation}
    Furthermore, for all $j=1,\cdots,M$, $P_\beta M_{\overline{\mu}}^*r_1 = b_j^{\top}\beta$, for any $b_j$ that solves
    \begin{equation}\label{eq:beta_projection}
        G_\beta b_j = \begin{bmatrix}
            \left\langle M_{\overline{\mu}}^*r_j,\beta_1 \right\rangle_{H}\\\vdots\\\left\langle M_{\overline{\mu}}^*r_j,\beta_M \right\rangle_{H}
        \end{bmatrix} = \begin{bmatrix}
            \left\langle \varpi_j,\beta_1 \right\rangle_{H}\\\vdots\\\left\langle \varpi_j,\beta_M \right\rangle_{H}
        \end{bmatrix}.
    \end{equation}
    As a result, $ b\coloneqq\begin{bmatrix} b_1, &\ldots,& b_M \end{bmatrix}$ is a solution of
    \begin{equation}
        G_\beta b = \left(\left\langle \varpi_j,\beta_i \right\rangle_{H}\right)_{i,j=1}^M
        = I^\top.\label{eq:b_matrix}
    \end{equation}
    Substituting \eqref{eq:beta_projection} into \eqref{eq:alpha_projection},
    \begin{equation*}
        G_r a = \begin{bmatrix}
            \left\langle A_{f,g}h,b_1^\top\beta\right\rangle_{H}\\\vdots\\\left\langle  A_{f,g}h,b_M^\top\beta\right\rangle_{H}
        \end{bmatrix} = \begin{bmatrix}
            \left\langle h,b_1^\top A_{f,g}^*\beta\right\rangle_{\tilde{H}_{d}}\\\vdots\\\left\langle  h,b_M^\top A_{f,g}^*\beta\right\rangle_{\tilde{H}_{d}}
        \end{bmatrix},
    \end{equation*}
    where $A_{f,g}^*\beta$ is interpreted as $A_{f,g}^*\beta = \begin{bmatrix}
        A_{f,g}^*\beta_1,&\ldots,&A_{f,g}^*\beta_M
    \end{bmatrix}^\top$. Using $A_{f,g}^* \beta_j = d_j$ and $ h = \delta^\top d$, 
    \begin{equation*}
        G_r a = \begin{bmatrix}
            \left\langle \delta^\top d,b_1^\top d\right\rangle_{\tilde{H}_{d}}\\\vdots\\\left\langle  \delta^\top d,b_M^\top d\right\rangle_{\tilde{H}_{d}}
        \end{bmatrix}
        = \begin{bmatrix}b_1^\top G_d d\\\vdots\\b_M^\top G_d d
        \end{bmatrix}
    \end{equation*}
    Selecting solutions of \eqref{eq:alpha_projection} and  \eqref{eq:b_matrix} that minimize the 2-norm of $a$ and $b_j$, respectively,
    \begin{equation}
        a = G_r^{+} b^\top G_d \delta = G_r^{+} I G_\beta^{+} G_d \delta = G_r^{+} I G_\beta^{+} G_d \delta.
    \end{equation}
    That is, a matrix representation $[M_{\overline{\mu}} P_\beta A_{f,g}]_d^r$ of the operator $P_r M_{\overline{\mu}} P_\beta A_{f,g}|_d$ is given by $G_r^{+} I G_\beta^{+} G_d$.
\end{proof}
Note that matrix representations are generally not unique. Different representations may be obtained by selecting different solutions of \eqref{eq:alpha_projection} and \eqref{eq:b_matrix}. In the case where the Gram matrices $G_r$ and $G_{\beta}$ are nonsingular, equations \eqref{eq:alpha_projection} and \eqref{eq:b_matrix} have unique solutions, resulting in the unique matrix representation $G_r^{-1} I G_\beta^{-1} G_d$.

In the following section, the matrix representation $[M_{\overline{\mu}} P_\beta A_{f,g}]_d^r$ is used to construct a data-driven representation of the singular values and the left and right singular functions of $P_r M_{\overline{\mu}} P_\beta A_{f,g}\vert_d$. 

\subsection{Singular Functions of the Finite-rank Operator}
Recall that the tuples $\{(\sigma_i,\phi_i,\psi_i)\}_{i=1}^M$, with $\sigma_i\in \mathbb{R}^n$, $\phi_i\in \tilde{H}_{d}$, and $\psi_i\in \tilde{H}_{r}$, are singular values, left singular vectors, and right singular vectors of $P_r M_{\overline{\mu}} P_\beta A_{f,g}|_d$, respectively, if $\forall h\in\vspan{d}$, $ P_r M_{\overline{\mu}} P_\beta A_{f,g} h = \sum_{i=1}^M \sigma_i \psi_i \left\langle h,\phi_i\right\rangle_{\tilde{H}_{d}}$. The following proposition states that the SVD of $P_r M_{\overline{\mu}} P_\beta A_{f,g}|_d$ can be computed using matrices in the matrix representation $[M_{\overline{\mu}} P_\beta A_{f,g}]_d^r$ developed in the previous section.
\begin{proposition}
    If $ (W,\Sigma,V) $ is the SVD of $G_r^{+} I G_\beta^{+}$ with $W = \begin{bmatrix} w_1,&\ldots,&w_M \end{bmatrix}$, $V = \begin{bmatrix} v_1,&\ldots,&v_M \end{bmatrix}$, and $\Sigma = \diag\left(\begin{bmatrix} \sigma_1,&\ldots,&\sigma_M \end{bmatrix}\right)$, then for all $i=1,\ldots,M$, $\sigma_i$ are singular values of $P_r M_{\overline{\mu}} P_\beta A_{f,g}|_d$ with left singular functions $\phi_i := v_i^\top d$ and right singular functions $\psi_i := w_i^\top r$.
\end{proposition}
\begin{proof}
    Let $\phi_i = v_i^\top d$ and $\psi_i = w_i^\top r$ and $h = \delta^\top d$. Then, 
\begin{equation*}
    P_r M_{\overline{\mu}} P_\beta A_{f,g} h = \sum_{i=1}^M \sigma_i \psi_i \left\langle h,\phi_i\right\rangle_{\tilde{H}_{d}}
    \iff P_r M_{\overline{\mu}} P_\beta A_{f,g} \delta^\top d = \sum_{i=1}^M \sigma_i w_i^\top r \left\langle \delta^\top d, v_i^\top d\right\rangle_{\tilde{H}_{d}}
\end{equation*}
Using the finite-rank representation, the collection $\{(\sigma_i,\phi_i,\psi_i)\}_{i=1}^M$, is an SVD of $P_r M_{\overline{\mu}} P_\beta A_{f,g}|_d$, if for all $\delta\in\mathbb{R}^M$,
\begin{equation}\label{eq:suff_cond_SVD}
    \left(G_r^{+} I G_\beta^{+} G_d \delta\right)^\top r = \left(\sum_{i=1}^M \sigma_i \left\langle \delta^\top d, v_i^\top d\right\rangle_{\tilde{H}_{d}} w_i^\top \right)r.
\end{equation}
Simple matrix manipulations yield the chain of implications
\begin{gather*}
    {\thickmuskip=0mu\thinmuskip=0mu\medmuskip=0mu\eqref{eq:suff_cond_SVD}\impliedby \forall \delta\in\mathbb{R}^M,G_r^{+} I G_\beta^{+} G_d \delta = \sum_{i=1}^M \sigma_i \left\langle \delta^\top d, v_i^\top d\right\rangle_{\tilde{H}_{d}} w_i}
    \iff \forall \delta\in\mathbb{R}^M,G_r^{+} I G_\beta^{+} G_d \delta 
    = \sum_{i=1}^M \sigma_i  \left(w_i v_i^\top G_d\right) \delta\\
    \impliedby G_r^{+} I G_\beta^{+} G_d = \sum_{i=1}^M \sigma_i  \left(w_i v_i^\top \right)G_d
    \impliedby G_r^{+} I G_\beta^{+} = \sum_{i=1}^M \sigma_i w_i v_i^\top = W\Sigma V^\top,
\end{gather*}
which proves the proposition.
\end{proof}
In the following section, the singular values and the left and right singular vectors are used, along with a finite truncation of \eqref{eq:infinite_spectral_reconstruction_svd} to generate a data-driven model.

\subsection{The SCLDMD Algorithm}
Motivated by \eqref{eq:total_derivative_model}, assuming that $h_{\mathrm{id},j} \in \tilde{H}_{d}$ for $j=1,\cdots,n$, the system dynamics are approximated using the rank-$M$ representation as $\dot{x} \approx \hat{F}_{\mu,M} (x) := [P_{r} M_{\overline{\mu}} P_{\beta} A_{f,g} P_{d} h_{\mathrm{id}}] (x)$, where $P_{r} M_{\overline{\mu}} P_{\beta} A_{f,g} P_{d} h_{\mathrm{id}}$ denotes row-wise operation of the operator $P_{r} M_{\overline{\mu}} P_{\beta} A_{f,g} P_{d}$ on the function $h_{\mathrm{id}}$. Since $P_{r} M_{\overline{\mu}} P_{\beta} A_{f,g} P_{d}$ converges to $M_{\overline{\mu}} A_{f,g}$ in norm as $M\to\infty$, and since the space $F^2_{\tilde{\rho}_d}(\mathbb{R}^n)$ contains $h_{\mathrm{id},j}$ for $j=1,\cdots,n$, the following result is immediate.
\begin{corollary}\label{cor:uniform-convergence}
    Under the hypothesis of Theorem \ref{thm:norm-convergence},  $\lim_{M\to\infty}\left(\sup_{x\in X}\left\Vert \hat{F}_{\mu,M}(x) - F_\mu(x) \right\Vert_{2}\right) = 0$.
\end{corollary}
\begin{proof}
   Since the space $F^2_{\tilde{\rho}_d}(\mathbb{R}^n)$ contains $h_{\mathrm{id},j}$ for $j=1,\cdots,n$, the functions $\hat{F}_{\mu,M,j} \coloneqq P_{r^j} M_{\overline{\mu}} P_{\beta^j} A_{f,g} P_{d^j} h_{\mathrm{id},j} $ and $F_{\mu,j} \coloneqq M_{\overline{\mu}} A_{f,g} h_{\mathrm{id},j} $ that denote the $j-$th row of $\hat{F}_{\mu,M}$ and $F_{\mu}$, respectively, exist as members of $\tilde{H}_{r}$. Since  $x\mapsto \tilde{K}_r(x,x) = \exp\left(\frac{x^\top x}{\tilde{\rho}_r}\right)$ is continuous and $X$ is compact, there exists a real number $\overline{K}$ such that $\sup_{x\in X}\tilde{K}_r(x,x) = \overline{K}$. Theorem \ref{thm:norm-convergence} can then be used to conclude that for all $\epsilon > 0$ and $j=1,\ldots,n$, there exists $M(j) \in \mathbb{N}$ such that for all $i\geq M(j)$, $\left\Vert \hat{F}_{\mu,i,j} - F_{\mu,j}\right\Vert_{\tilde{H}_{r}}^{2} \leq \frac{\epsilon^2}{n\overline{K}^2}$. Using the reproducing property, for $i\geq \overline{M} \coloneqq \max_j M(j)$, 
\begin{multline*}
    \left\Vert \hat{F}_{\mu,i}(x) - F_\mu(x) \right\Vert_{2}^{2} = \sum_{j=1}^{n} \left\langle \left(\hat{F}_{\mu,i,j} - F_{\mu,j}\right),\tilde{K}_r(\cdot,x)\right\rangle_{\tilde{H}_{r}}^{2}
    \leq \sum_{j=1}^{n} \left\Vert \hat{F}_{\mu,i,j} - F_{\mu,j} \right\Vert_{\tilde{H}_{r}}^{2}\left\Vert \tilde{K}_r(\cdot,x) \right\Vert_{\tilde{H}_{r}}^{2}\\
    \leq \sum_{j=1}^{n} \frac{\epsilon^2}{n\overline{K}^2}\left\langle \tilde{K}_r(\cdot,x),\tilde{K}_r(\cdot,x) \right\rangle_{\tilde{H}_{r}}^{2} = \frac{\epsilon^2}{\overline{K}^2}\tilde{K}_r(x,x)^{2}.
\end{multline*}
As a result, for all $\epsilon \geq 0$  there exists $\overline{M}$ such that for all $i\geq \overline{M}$,
\begin{equation*}
    \sup_{x\in X}\left\Vert \hat{F}_{\mu,i}(x) - F_\mu(x) \right\Vert_{2} \leq \sqrt{\frac{\epsilon^2}{\overline{K}^2}\sup_{x\in X}\tilde{K}_r(x,x)^{2}} = \epsilon,
\end{equation*}
which completes the proof.
\end{proof}
Using the definition of singular values and singular functions,
\begin{equation}
    \dot{x} \approx
    \sum_{i=1}^M \sigma_i \xi_i w_i^\top r (x) = \xi\Sigma W^\top r(x),
\end{equation}
where $\xi_i \coloneqq \left\langle P_d h_{\mathrm{id}},\phi_i\right\rangle_{\tilde{H}_{d}}$ and $\xi := \begin{bmatrix}
    \xi_1,&\ldots,&\xi_M
\end{bmatrix}$.

The modes $\xi$ can be computed using $\phi_i = v_i^\top d$ as
\begin{gather*}
    \xi = \begin{bmatrix}
    \left\langle P_d h_{\mathrm{id},1},v_1^\top d\right\rangle_{\tilde{H}_{d}},&\ldots,&\left\langle P_d h_{\mathrm{id},1},v_M^\top d\right\rangle_{\tilde{H}_{d}}\\
    \vdots & \ddots & \vdots\\
\left\langle P_d h_{\mathrm{id},n},v_1^\top d\right\rangle_{\tilde{H}_{d}},&\ldots,&\left\langle P_d h_{\mathrm{id},n},v_M^\top d\right\rangle_{\tilde{H}_{d}}
\end{bmatrix}
= \begin{bmatrix}
    \left\langle \delta_1^\top d,d_1\right\rangle_{\tilde{H}_{d}},&\ldots,&\left\langle \delta_1^\top d,d_M\right\rangle_{\tilde{H}_{d}}\\
    \vdots & \ddots & \vdots\\
\left\langle \delta_n^\top d,d_1\right\rangle_{\tilde{H}_{d}},&\ldots,&\left\langle \delta_n^\top d,d_M\right\rangle_{\tilde{H}_{d}}
\end{bmatrix} V
= \delta^\top G_d V,
\end{gather*}
where $\delta \coloneqq  \begin{bmatrix} \delta_1, &\ldots, &\delta_n \end{bmatrix}$. Using the reproducing property of the reproducing kernel of $\tilde{H}_d$, the coefficients $\delta_i$ in the projection of $h_{\mathrm{id},i}$ onto $d$ satisfy
{\thinmuskip=0mu \thickmuskip=0mu \medmuskip=0mu \begin{equation*}
    G_d \delta_i = \begin{bmatrix}
        \left\langle\left(h_{\mathrm{id}}\right)_i,d_1\right\rangle_{\tilde{H}_{d}}\\ \vdots \\ \left\langle\left(h_{\mathrm{id}}\right)_i,d_M\right\rangle_{\tilde{H}_{d}}
    \end{bmatrix} = \begin{bmatrix}
        \left(\gamma_{u_1}(T_1)\right)_i - \left(\gamma_{u_1}(0)\right)_i\\ \vdots \\ \left(\gamma_{u_M}(T_M)\right)_i - \left(\gamma_{u_M}(0)\right)_i
    \end{bmatrix}.
\end{equation*}}
Letting $D \coloneqq \left(\left(\gamma_{u_j}(T_j)\right)_i - \left(\gamma_{u_j}(0)\right)_i\right)_{i,j=1}^{n,M} $ it can be concluded that $ \delta^\top G_d = D$. Finally, the modes $\xi$ are given by $\xi = D V$ and the estimated closed-loop model is given by
\begin{equation}\label{eq:convergent_closed_loop_model}
    \dot{x} \approx \hat{F}_{\mu,M}(x) = D V\Sigma W^\top r(x) = D G_\beta^{+} I^\top G_r^{+} r(x)
\end{equation}
The SCLDMD technique is summarized in Algorithm \ref{alg:SCLDMD}. The characterization $ \Gamma_{\gamma_{u_j}} = \int_0^{T_j} \tilde{K}\left(\cdot,\gamma_{u_j}(t)\right) \mathrm{d}t$ of occupation kernels, introduced in \cite{SCC.Rosenfeld.Russo.eatoappear}, is used on line \ref{line:psi}.
\begin{algorithm}
    \caption{\label{alg:SCLDMD}The SCLDMD algorithm}
    \begin{algorithmic}[1]
        \renewcommand{\algorithmicrequire}{\textbf{Input:}}
        \renewcommand{\algorithmicensure}{\textbf{Output:}}
        \REQUIRE Trajectories $\{\gamma_{u_i}\}_{i=1}^{M}$, a feedback law $\mu$, a numerical integration procedure, reproducing  kernels $\tilde{K}_d$, $\tilde{K}_r$, and $K$ of $\tilde{H}_d$, $\tilde{H}_r$, and $H$, respectively, and regularization parameters $\epsilon_r$ and $\tilde{\epsilon}$.
        \ENSURE $\{\xi_j,\sigma_j,\varphi_j,\phi_j\}_{j=1}^{M}$
        \STATE $G_\beta \leftarrow \left(\left\langle \Gamma_{\gamma_{u_i},u_i}, \Gamma_{\gamma_{u_j},u_j} \right\rangle_{H} \right)_{i,j=1}^M$ (see \eqref{eq:Control_occ_ker_gram_matrix})
        \STATE $G_r \leftarrow  \left( \left\langle\Gamma_{\gamma_{u_i}},\Gamma_{\gamma_{u_j}}\right\rangle_{\tilde{H}_r}\right)_{i,j=1}^{M}$ (see \eqref{eq:compute_gram_matrix_occ_kernel})
        \STATE $ I \leftarrow \left(\left\langle \Gamma_{\gamma_{u_i},\mu\circ\gamma_{u_i}},\Gamma_{\gamma_{u_j},u_j}\right\rangle_H\right)_{i,j=1}^{M} $ (see \eqref{eq:int_matrix_occ_kernel})
        \STATE $ D \leftarrow \left( \left(\gamma_{u_j}(T_j)\right)_i - \left(\gamma_{u_j}(0)\right)_i\right)_{i,j=1}^{n,M}$
        \STATE $(W,\Sigma,V)\leftarrow$ SVD of $ G_r^{+} I G_\beta^{+} $ (See Remark \ref{rem:rank_deficient})
        \STATE $\xi \leftarrow DV$
        \STATE $\phi_j \leftarrow \sum_{i=1}^M\leftarrow (V)_{i,j} \left(K_d(\cdot,\gamma_{u_i}(T_i)) - K_d(\cdot,\gamma_{u_i}(0))\right)$
        \STATE $\psi_j \leftarrow \sum_{i=1}^M \int_0^{T_i} (W)_{i,j}\tilde{K}\left(\cdot,\gamma_{u_i}(t)\right) \mathrm{d}t$\label{line:psi}
        \RETURN $\{\xi_j,\sigma_j,\varphi_j,\phi_j\}_{j=1}^{M}$ 
    \end{algorithmic} 
\end{algorithm}

\section{Eigendecomposition Approach to DMD \label{sec:eig-DMD}}
In this section, an alternative finite-rank representation of the operator $M_{\overline{\mu}}A_{f,g}$ is presented, where its domain and range are assumed to be subsets of the same RKHS $\tilde{H}$ of complex-valued continuously differentiable functions, with a real-valued reproducing kernel $\tilde{K}$. In particular, the finite-rank representation of $ M_{\overline{\mu}}A_{f,g} $ is selected to be $P_r M_{\overline{\mu}} P_\beta A_{f,g}|_r$, where the domain and the range are both $\vspan{r}$. A consequence of this choice is that the finite-rank representation admits eigenfunctions which could potentially generate an approximate invariant subspace of the closed-loop system. 

While eigenfunctions of the finite-rank representation exist, they generally cannot be shown to converge to eigenfunctions of the original operator, since the operators $M_{\overline{\mu}}$ and $A_{f,g}$ can no longer be assumed to be bounded and compact, respectively. Instead, they are assumed to be densely defined. Since the operators are not defined everywhere, we need the additional assumptions that 1) the image of $A_{f,g}$ is contained within the domain of $M_{\overline{\mu}}$, 2) the span of $r$ is a subset of the domain of $A_{f,g}$, and 3) the functions $h_{\mathrm{id},j}$ can be well-approximated by linear combinations of the eigenfunctions of the finite-rank representation for $j=1,\ldots,n$. Due to the lack of convergence guarantees and since the assumptions on $\vspan r$ and $h_{\mathrm{id},j}$ are difficult to verify, the resulting algorithm, while useful, is heuristic in nature. Since unbounded operators over Hilbert spaces of real-valued functions can have empty spectra, in this section, the RKHS $\tilde{H}$ is assumed to be composed of complex-valued functions of real variables of the form $h:X \to \mathbb{C}$.

The operators $M_{\overline{\mu}}$ and $A_{f,g}$  are densely defined in a large class of problems. For example, if the domain and range spaces in Section \ref{subsec:existence-compactness} are selected to have identical kernel parameters, then the resulting operators are densely defined \cite{SCC.Rosenfeld.Russo.eatoappear}, and the image of $A_{f,g}$ is also contained within the domain of $M_{\overline{\mu}}$. The assumption that $\vspan r \subset \mathcal{D}(A_{f,g})$ can be removed in favor of the assumption that the matrix that encodes the finite rank representation of $M_{\overline{\mu}}A_{f,g}$ is approximately equal to the transpose of the matrix that encodes the finite rank representation of the adjoint $ A_{f,g}^* M_{\overline{\mu}}^* $ (see \cite{SCC.Rosenfeld.Kamalapurkar2021}).

\subsection{Matrix Representation of the Finite-rank Operator}
In this section, a matrix representation of the finite-rank representation is developed.
\begin{proposition}
    If $ h = \delta^\top r\in \vspan{r}$ is a function with coefficients $\delta\in\mathbb{C}^{M}$, $A_{f,g}$ and $M_{\overline{\mu}}$ are densely defined, $\vspan r \subset \mathcal{D}(A_{f,g})$, $\vspan \beta \subset \mathcal{D}(M_{\overline{\mu}})$, and $g = P_r M_{\overline{\mu}} P_\beta A_{f,g} h$, then $g = a^\top r$ with $a = G_{r}^{+}IG_{\beta}^{+}\tilde{I}^{\top}\delta$.
\end{proposition} 
\begin{proof}
The coefficients $a = \begin{bmatrix}
a_1&\cdots&a_{M}
\end{bmatrix}^{\top} \in \mathbb{C}^M$ in the projection of $ M_{\overline{\mu}} P_\beta A_{f,g}h \in \tilde{H}$ onto $\vspan r$ are given by the solution of the linear system
\begin{equation}
    G_{r} a = \begin{bmatrix} \left\langle M_{\overline{\mu}} P_\beta A_{f,g}h,r_{1}\right\rangle_{\tilde{H}} \\ \vdots \\  \left\langle M_{\overline{\mu}} P_\beta A_{f,g}h,r_{M}\right\rangle_{\tilde{H}} \end{bmatrix}\label{eq:Gram_alpha}.
\end{equation}
A matrix representation of $P_r M_{\overline{\mu}} P_\beta A_{f,g}\vert_r$ relates the coefficients $\delta = \begin{bmatrix}
\delta_1&\cdots&\delta_{M}
\end{bmatrix}^{\top}  \in \mathbb{C}^M$ of a function $ h = \delta^\top r\in\vspan r $, with the coefficients $a$ above. Using the properties of the multiplication operator and the control Liouville operator established in the previous sections, the inner products $\left\langle M_{\overline{\mu}} P_\beta A_{f,g}h,r_j \right\rangle_{\tilde{H}}$ on the right hand side can be evaluated as
\begin{equation*}
    \left\langle M_{\overline{\mu}} P_\beta A_{f,g}h,r_j \right\rangle_{\tilde{H}} = \sum_{i=1}^{M}\delta_i\left\langle A_{f,g}r_i,P_\beta M_{\overline{\mu}}^{*} r_j \right\rangle_{H}
    =\sum_{i=1}^{M}\delta_i\left\langle A_{f,g}r_i,\sum_{k=1}^{M}b_{k,j}\Gamma_{\gamma_{u_k},u_k} \right\rangle_{H}.
\end{equation*}
where $\left\{b_{k,j}\right\}_{k=1}^{M}\subset\mathbb{C}$ are the coefficients in the projection of $M_{\overline{\mu}}^{*} r_j \in H$ onto $\vspan\beta$, which can be computed by solving
\begin{equation}
G_{\beta} \begin{bmatrix} b_{1,j} \\ \vdots \\ b_{M,j} \end{bmatrix} = \begin{bmatrix} \left\langle M_{\overline{\mu}}^{*} r_j, \Gamma_{\gamma_{u_1},u_1}\right\rangle_H \\ \vdots \\ \left\langle M_{\overline{\mu}}^{*} r_j, \Gamma_{\gamma_{u_M},u_M} \right\rangle_H \end{bmatrix}\label{eq:Gram_beta}.
\end{equation}
Note that since the control occupation kernels $\beta_i=\Gamma_{\gamma_{u_i},u_i}$ the occupation kernels $r_i=\Gamma_{\gamma_{u_i}}$, and the symbol $\overline{\mu}$ are all real-valued functions, the coefficients $b_{i,j}$ are real numbers. The inner product can thus be further simplified as
\begin{equation*}
    \left\langle M_{\overline{\mu}} P_\beta A_{f,g}h,r_j \right\rangle_{\tilde{H}}=\sum_{i=1}^{M}\delta_i\sum_{k=1}^{M}b_{k,j}\left\langle r_i,A_{f,g}^{*}\Gamma_{\gamma_{u_k},u_k} \right\rangle_{\tilde{H}}
    =\delta^{\top}\tilde{I}b_j,
\end{equation*}
where $b_j \coloneqq \begin{bmatrix} b_{1,j}&\cdots&b_{M,j} \end{bmatrix}^{\top}$, and $\tilde{I} \coloneqq \left(\left\langle r_i,A_{f,g}^{*}\Gamma_{\gamma_{u_k},u_k} \right\rangle_{\tilde{H}}\right)_{i,k=1}^{M}$ is the interaction matrix corresponding to $\tilde{H}$. Stacking the inner products on the left hand side in a column and selecting solutions of \eqref{eq:Gram_alpha}  and \eqref{eq:Gram_beta} that minimize the 2-norm of $a$ and $b_j$, respectively, it can be concluded that $ a = G_{r}^{+}IG_{\beta}^{+}\tilde{I}^{\top}\delta $, where  $ I \coloneqq \left(\left\langle M_{\overline{\mu}}^{*}r_j,\Gamma_{\gamma_{u_k},u_k}\right\rangle_H\right)_{j,k=1}^{M} $ is the interaction matrix corresponding to $H$. A matrix representation $[M_{\overline{\mu}} P_\beta A_{f,g}]_r^r \in\mathbb{R}^M$ of the finite-rank representation $P_r M_{\overline{\mu}} P_\beta A_{f,g}\vert_r$ of the operator $M_{\overline{\mu}} A_{f,g}$ is thus given by $ [M_{\overline{\mu}} P_\beta A_{f,g}]_r^r = G_{r}^{+} I G_{\beta}^{+}\tilde{I}^{\top}$.
\end{proof}
In the following section, the matrix representation $[M_{\overline{\mu}} P_\beta A_{f,g}]_r^r$ is used to construct a data-driven representation of the eigenvalues and the eigenfunctions of $P_r M_{\overline{\mu}} P_\beta A_{f,g}\vert_r$. 
\subsection{Eigenfunctions of the finite-rank representation\label{subsec:CLDMD}}
Given an eigenvalue $\tilde{\lambda}_j\in\mathbb{C}$ and the corresponding eigenvector $\tilde{v}_j := \begin{bmatrix}
\tilde{v}_{1,j}&\cdots&\tilde{v}_{M,j}
\end{bmatrix}^{\top}\in\mathbb{C}^M$ of $[M_{\overline{\mu}} P_\beta A_{f,g}]_r^r$ and the vector $r := \begin{bmatrix}
r_1&\cdots&r_{M}
\end{bmatrix}^{\top}$ of occupation kernels in $\tilde{H}$, it is straightforward to show that $\varphi_j = \left(\nicefrac{1}{\sqrt{\tilde{v}_j^\dagger G_{r} \tilde{v}_j}}\right) \tilde{v}_{j}^{\top}r$ is an eigenfunction of $P_r M_{\overline{\mu}} P_\beta A_{f,g}\vert_r$, where $(\cdot)^\dagger$ denotes the conjugate transpose. Indeed, by the definition of the matrix $[M_{\overline{\mu}} P_\beta A_{f,g}]_r^r$, it can be seen that $ P_r M_{\overline{\mu}} P_\beta A_{f,g}\vert_{r}\varphi_j =  \left(\nicefrac{1}{\sqrt{\tilde{v}_j^\dagger G_{r} \tilde{v}_j}}\right)\left([M_{\overline{\mu}} P_\beta A_{f,g}]_r^r \tilde{v}_j\right)^{\top}r = \tilde{\lambda}_j\left(\nicefrac{1}{\sqrt{\tilde{v}_j^\dagger G_{r} \tilde{v}_j}}\right) \tilde{v}_j^{\top}r $.

Using the fact that $ r_i(x) = \Gamma_{\gamma_{u_i}}(x) = \int_0^{T_i} \tilde{K}(x,\gamma_{u_i}(t)) \mathrm{d}t $, the eigenfunctions, evaluated at a point $x\in\mathbb{R}^n$, can be computed as
\begin{equation}
    \varphi_j(x) = \frac{1}{\sqrt{\tilde{v}_j^\dagger G_{r} \tilde{v}_j}}\sum_{i=1}^{M}\tilde{v}_{i,j}\int_0^{T_i} \tilde{K}\left(x,\gamma_{u_i}(t)\right) \mathrm{d}t.\label{eq:occ_kernel_eigenfunctions}
\end{equation}
In the following section, the eigenvalues and the eigenfunctions are used to generate a data-driven model.
\subsection{The CLDMD Algorithm}
Let $\tilde{W} = \left(\nicefrac{\tilde{v}_{i,j}}{\sqrt{\tilde{v}_j^\dagger G_{r} \tilde{v}_j}}\right)_{i,j=1}^{M}\in\mathbb{C}^{M\times M}$ be the matrix of coefficients of the normalized eigenfunctions, arranged so that each column corresponds to an eigenfunction. Assuming that $h_{\mathrm{id},j}$ is in the span of the above eigenfunctions for each $j=1,\cdots, n$, a representation of the identity function as a linear combination of a fixed number of eigenfunctions is given as $h_{\mathrm{id}}(x) \approx \sum_{i=1}^{M} \xi_i \varphi_i(x)$, where $ \{\xi_i\}_{i=1}^{M}\subset\mathbb{C}^{n} $ are the so-called \emph{control-Liouville modes}. Similar to \cite[Section 4.2]{SCC.Rosenfeld.Kamalapurkar.ea2022}, by examining the inner products $\left\langle h_{\mathrm{id},j},r_i\right\rangle_{\tilde{H}}$, the matrix $\xi \coloneqq \begin{bmatrix}\xi_1,& \cdots,& \xi_{M}\end{bmatrix}$ can be shown to be a solution of the linear system of equations 
\begin{equation}
    \xi\left(\tilde{W}^{\top}G_{r}\overline{\tilde{W}}\right) = R\overline{\tilde{W}}
    ,\label{eq:modes_lin_sys}
\end{equation}
where $R \coloneqq \left(\left\langle h_{\mathrm{id},j},r_i \right\rangle_{\tilde{H}}\right)_{j,i=1}^{n,M}$ and  $\overline{\tilde{W}}$ denotes the complex conjugate of $\tilde{W}$. Indeed, letting $\xi_{i,j}$ denote the $j-$th element of the vector $\xi_i$, the row of coefficients $\xi^j\coloneqq \begin{bmatrix}
    \xi_{1,j},& \ldots,\xi_{M,j}
\end{bmatrix}\in\mathbb{C}^{1\times M} $  in the projection of $h_{
\mathrm{id},j
}$ onto the span of the eigenfunctions $\{\varphi_i\}_{i=1}^M$ is a solution of
\begin{equation}
    \xi^j G_{\varphi}^{\top} = \begin{bmatrix}
        \left\langle h_{\mathrm{id},j},\varphi_1 \right\rangle_{\tilde{H}},&
        \ldots,&
        \left\langle h_{\mathrm{id},j},\varphi_M \right\rangle_{\tilde{H}}
    \end{bmatrix},\label{eq:ModesRowWise}
\end{equation}
where $G_{\varphi} = \left(\left\langle \varphi_i,\varphi_j\right\rangle_{\tilde{H}}\right)_{i,j=1}^M \in \mathbb{C}^{M\times M}$. Using the fact that $\left\langle \varphi_i,\varphi_j\right\rangle_{\tilde{H}} = \left\langle \tilde{w}_{i}^{\top}r,\tilde{w}_{j}^{\top}r\right\rangle_{\tilde{H}} = \tilde{w}_j^{\dagger} G_r \tilde{w}_i$, where $\tilde{w}_{j} \coloneqq \left(\nicefrac{1}{\sqrt{\tilde{v}_j^\dagger G_{r} \tilde{v}_j}}\right)\tilde{v}_j $, denotes the $j-$th column of $\tilde{W}$, the Gram matrix $G_\varphi$ can be expressed as $G_\varphi = \tilde{W}^\dagger G_r \tilde{W}$. Furthermore, using the fact that $\left\langle h_{\mathrm{id},j},\varphi_i \right\rangle_{\tilde{H}} = \left\langle h_{\mathrm{id},j},\tilde{w}^{\top}_i r \right\rangle_{\tilde{H}} = \begin{bmatrix}
    \left\langle h_{\mathrm{id},j},r_1 \right\rangle_{\tilde{H}},&\ldots,& \left\langle h_{\mathrm{id},j},r_M \right\rangle_{\tilde{H}}
\end{bmatrix} \overline{\tilde{w}_i}$, where $\overline{\tilde{w_i}}$ denotes the complex conjugate of $\tilde{w}_i$, the right hand side of \eqref{eq:ModesRowWise} can be expressed as $\begin{bmatrix}
        \left\langle h_{\mathrm{id},j},\varphi_1 \right\rangle_{\tilde{H}},&
        \ldots,&
        \left\langle h_{\mathrm{id},j},\varphi_M \right\rangle_{\tilde{H}}
    \end{bmatrix} = \begin{bmatrix}
    \left\langle h_{\mathrm{id},j},r_1 \right\rangle_{\tilde{H}},&\ldots,& \left\langle h_{\mathrm{id},j},r_M \right\rangle_{\tilde{H}}
\end{bmatrix} \overline{\tilde{W}}$. Concatenating \eqref{eq:ModesRowWise} for $j=1,\ldots,n$ into a column vector, the matrix $\xi$ is seen to be a solution of \eqref{eq:modes_lin_sys}.

Using the fact that any solution of $\xi \tilde{W}^\top G_r = R$ is also a solution of \eqref{eq:modes_lin_sys}, selecting the solution of $\xi \tilde{W}^\top G_r = R$ that minimizes the 2-norm of $\xi_i$ for $i=1,\cdots,M$, and using the relationship $ \left\langle h_{\mathrm{id},j},r_i \right\rangle_{\tilde{H}} = \left\langle h_{\mathrm{id},j},\Gamma_{\gamma_{u_i}} \right\rangle_{\tilde{H}} = \int_{0}^{T_i} \gamma_{u_i,j}(t)\mathrm{d}t $, where $\gamma_{u_i,j}(t)$ denotes the $j-$th component of $\gamma_{u_i}(t)$, a set of control Liouville modes can be obtained as
\begin{equation}
    \xi = \begin{bmatrix}
    \int_{0}^{T_1} \gamma_{u_1}(t)\mathrm{d}t & \cdots & \int_{0}^{T_{M}} \gamma_{u_{M}}(t)\mathrm{d}t
\end{bmatrix}\left(\tilde{W}^{\top}G_{r}\right)^{+}.\label{eq:occ_kernel_modes}
\end{equation} 
The response $t\mapsto\gamma_{\mu}(t)$ of the system, starting from the initial condition $\gamma_{\mu}(0)=\gamma_0$, under the feedback control law $\mu$, can then be predicted as
\begin{equation}
    \gamma_{\mu}(t) \approx \sum_{j=1}^{M} \xi_j \varphi_j(\gamma_0) \mathrm{e}^{\tilde{\lambda}_j t}.\label{eq:finite_spectral_reconstruction}
\end{equation}

Furthermore, a pointwise approximation of the closed-loop model can also be obtained as
\begin{equation}
    \dot{x} \approx \hat{F}_{\mu,M}(x) \coloneqq \sum_{j=1}^{M} \tilde{\lambda}_j \xi_j \varphi_j(x).\label{eq:vector_field_reconstruction}
\end{equation}
The CLDMD method is summarized in Algorithm \ref{alg:CLDMD}.
\begin{algorithm}
    \caption{\label{alg:CLDMD}The CLDMD algorithm}
    \begin{algorithmic}[1]
        \renewcommand{\algorithmicrequire}{\textbf{Input:}}
        \renewcommand{\algorithmicensure}{\textbf{Output:}}
        \REQUIRE Trajectories $\{\gamma_{u_i}\}_{i=1}^{M}$, a feedback law $\mu$, a numerical integration procedure, Reproducing  kernel $\tilde{K}$ of $\tilde{H}$, Reproducing  kernel $K$ of $H$, and if needed, regularization parameters $\epsilon$ and $\tilde{\epsilon}$.
        \ENSURE $\{\xi_j,\lambda_j,\varphi_j\}_{j=1}^{M}$
        \STATE $G_{\beta} \leftarrow \left(\left\langle \Gamma_{\gamma_{u_i},u_i}, \Gamma_{\gamma_{u_j},u_j} \right\rangle_{H}\right)_{i,j=1}^M$ (see \eqref{eq:Control_occ_ker_gram_matrix})
        \STATE $G_{r} \leftarrow  \left( \left\langle r_i,r_j\right\rangle_{\tilde{H}}\right)_{i,j=1}^{M}$ (see 
        \eqref{eq:compute_gram_matrix_occ_kernel})
        \STATE $ I \leftarrow \left(\left\langle M_{\overline{\mu}}^{*}r_j,\Gamma_{\gamma_{u_k},u_k}\right\rangle_H\right)_{j,k=1}^{M} $ (see 
        \eqref{eq:int_matrix_occ_kernel})
        \STATE $\tilde{I} \leftarrow \left(\left\langle r_i,A_{f,g}^{*}\Gamma_{\gamma_{u_k},u_k} \right\rangle_{\tilde{H}}\right)_{i,k=1}^{M}$ (see 
        \eqref{eq:int_tilde_matrix_occ_kernel})
        \STATE $[M_{\overline{\mu}} P_\beta A_{f,g}]_r^r \leftarrow G_{r}^{+} I G_{\beta}^{+}\tilde{I}^{\top}$ (See Remark \ref{rem:rank_deficient})
        \STATE $\{\lambda_j,\tilde{v}_j\}_{j=1}^{M}\leftarrow $ eigendecomposition of $ [M_{\overline{\mu}} P_\beta A_{f,g}]_r^r $
        \STATE $\tilde{W} \leftarrow \left(\nicefrac{\tilde{v}_{i,j}}{\sqrt{\tilde{v}_j^\dagger G_{r} \tilde{v}_j}}\right)_{i,j=1}^{M}$
        \STATE Compute $\{\xi_j,\varphi_j\}_{j=1}^{M}$ using 
        \eqref{eq:occ_kernel_modes} and \eqref{eq:occ_kernel_eigenfunctions} 
        (See Remark \ref{rem:rank_deficient})
        \RETURN $\{\xi_j,\lambda_j,\varphi_j\}_{j=1}^{M}$ 
    \end{algorithmic} 
\end{algorithm}

\section{Computation of Inner Products}\label{sec:computations}
The elements of the Gram matrix $G_{\beta}$, corresponding to $\beta$, can be computed using Proposition \ref{prop:occ_ker_representation} as
\begin{equation}
    \left\langle \Gamma_{\gamma_{u_i},u_i}, \Gamma_{\gamma_{u_j},u_j} \right\rangle_{H}
    =\int\limits_{0}^{T_j}\int\limits_{0}^{T_i} \begin{bmatrix}1 & u_i^{\top}(\tau)\end{bmatrix}  K\left(\gamma_{u_j}(t),\gamma_{u_i}(\tau)\right)\begin{bmatrix}
    1 \\ u_j(t)
    \end{bmatrix}\mathrm{d}\tau \mathrm{d}t\label{eq:Control_occ_ker_gram_matrix}
\end{equation}
The elements of the Gram matrix $G_{r}$ can be computed using the double integral (cf. \cite{SCC.Rosenfeld.Kamalapurkar.ea2022})
\begin{equation}
    \left\langle\Gamma_{\gamma_{u_i}}, \Gamma_{\gamma_{u_j}}\right\rangle_{\tilde{H}} = \int\limits_{0}^{T_j}\int\limits_{0}^{T_i}\tilde{K}\left(\gamma_{u_j(t)},\gamma_{u_i(\tau)}\right)\mathrm{d}\tau\mathrm{d}t.\label{eq:compute_gram_matrix_occ_kernel}
\end{equation}
Using Proposition \ref{prop:adjoint_kernel_difference}, the elements of the interaction matrix $\tilde{I}$ can be evaluated as
\begin{equation}\hspace*{-1em}
    \left\langle \!\Gamma_{\gamma_{u_i}}\!,A_{f,g}^{*}\Gamma_{\gamma_{u_k},u_k} \!\right\rangle_{\tilde{H}}\!\!
    = \!\left\langle\! \Gamma_{\gamma_{u_i}}\!,\!\tilde{K}(\cdot,\gamma_{u_k}\!(T_k)\!)\! -\! \tilde{K}(\cdot,\gamma_{u_k}(0)\!) \!\right\rangle_{\tilde{H}}
    = \int\limits_{0}^{T_i}\left(\tilde{K}(\gamma_{u_i}(t),\gamma_{u_k}(T_k))-\tilde{K}(\gamma_{u_i}(t),\gamma_{u_k}(0))\right)\mathrm{d}t.\label{eq:int_tilde_matrix_occ_kernel}
\end{equation}
Using Proposition \ref{prop:mult_adjoint_feedback}, the elements of the interaction matrix $I$ can be evaluated as 
\begin{equation}
    \left\langle M_{\overline{\mu}}^{*}\Gamma_{\gamma_{u_i}},\Gamma_{\gamma_{u_k},u_k}\right\rangle_H = \left\langle \Gamma_{\gamma_{u_i},\mu\circ\gamma_{u_i}},\Gamma_{\gamma_{u_k},u_k}\right\rangle_H
    = \!\!\int\limits_{0}^{T_j}\!\int\limits_{0}^{T_i}\!\! \begin{bmatrix}1 & \mu^{\top}\!(\gamma_{u_i}(\tau)\!)\!)\end{bmatrix}\! K\!\left(\gamma_{u_j}(t),\gamma_{u_i}(\tau)\!\right)\!\!\begin{bmatrix}\!
    1 \\ u_j(t)
    \!\end{bmatrix}\!\mathrm{d}\tau \mathrm{d}t.\label{eq:int_matrix_occ_kernel}
\end{equation}
Assuming that each trajectory is sampled at $N$ points in time, the computation of $G_\beta$ and $I$ is $O(nN^2M^2(m+1)^2)$, the computation of $G_r$ is $O(nN^2M^2)$, and the computation of $\tilde{I}$ is $O(nNM^2)$. Computation of the finite-rank representation and its decomposition are $O(M^3)$. Evaluation of the occupation kernel is $O(nN)$. 

\begin{remark}\label{rem:rank_deficient}
    In addition to the Moore-Penrose pseudoinverse, the SCLDMD and CLDMD algorithms can also be implemented using regularization. Regularization involves replacing the Gram matrices $G_{\beta}$ and $G_{r}$ by $G_{\beta}+\epsilon \mathrm{I}_M$ and $G_{r} + \tilde{\epsilon} \mathrm{I}_{M}$, respectively, whenever they need to be inverted, where $\mathrm{I}_M$ denotes the $M\times M$ identity matrix, and $\epsilon>0$ and $\tilde{\epsilon}>0$ are user-selected regularization coefficients.
\end{remark}

\section{Numerical Experiments\label{sec:sims}}
Two numerical experiments are performed to evaluate the developed SCLDMD and CLDMD methods, one using a simulated controlled Duffing oscillator and another using a simulated two-link robot manipulator.
\subsection{Controlled Duffing oscillator\label{subsec:duffing}}
This experiment concerns the controlled Duffing oscillator
\begin{equation*}
    \dot{x}_1 = x_2,\qquad \dot{x}_2 = x_1 - x_1^3 + \left(2 + \sin(x_1)\right)u,
\end{equation*}
where $x=\begin{bmatrix} x_1&x_2 \end{bmatrix}^{\top}\in\mathbb{R}^2$ is the state and $u\in\mathbb{R}$ is the control. A total of 225 open-loop trajectories of the controlled Duffing oscillator are generated using the MATLAB\textsuperscript{\textregistered} \texttt{ode45} solver, starting from initial conditions on a $15\times 15$ regular grid on a $6 \times 6$ square centered at the origin of the state space, $\mathbb{R}^2$. The control signal used for trajectory generation is of the form $u(t) = \sum_{i=1}^{15} b_i\sin(\omega_i t + \varphi_i)$, where the magnitudes $b_i$, the frequencies $\omega_i$, and the phase differences $\varphi_i$ are generated randomly from a uniform distribution on the interval $[-1,1]$. All trajectories are recorded over a duration of $1$\si{\second}, and are sampled at a frequency of $20$ \si{\hertz}.

The trajectories are then utilized to predict the behavior of the oscillator under the state feedback controller $\mu(x) = \begin{bmatrix} -2&-2 \end{bmatrix}x$. CLDMD is implemented using the exponential dot product reproducing kernel $ \tilde{K}_{\tilde{\rho}} = \exp\left(\frac{x^\top y}{\tilde{\rho}}\right) $ with parameter $\tilde{\rho} = 5$, and a diagonal kernel given by $ K = \mathrm{diag}\begin{bmatrix} \tilde{K}_{\tilde{\rho}} & \tilde{K}_{\tilde{\rho}}\end{bmatrix} $. SCLDMD is implemented using $\tilde{K}_r = \tilde{K}_{5} $, $ K = \mathrm{diag}\begin{bmatrix} \tilde{K}_{6} & \tilde{K}_{6}\end{bmatrix} $, and $\tilde{K}_d = \tilde{K}_{7} $. Simpson's 1/3 rule is used to compute the integrals involved in algorithms \ref{alg:SCLDMD} and \ref{alg:CLDMD}.

\subsubsection{Vector Field Reconstruction}
Fig. \ref{fig:vector_field_error} shows a side by side comparison of the pointwise 2-norm of the relative error between the approximated vector field $\hat{F}_{\mu,M}$ (generated using \eqref{eq:convergent_closed_loop_model} for SCLDMD and \eqref{eq:vector_field_reconstruction} for CLDMD), and the true vector field, $F_{\mu}$. The results in Fig. \ref{fig:vector_field_error} indicate that both the CLDMD and the SCLDMD methods are able to obtain accurate estimates of the closed-loop vector field on a domain contained within the grid of initial conditions of the data.
\begin{figure}[ht]
    \centering
    \begin{tikzpicture}
    \begin{axis}[
        width=0.5\columnwidth,
        colormap/viridis,
        xlabel={$x_1$},
        ylabel={$x_2$},
        zlabel={Relative Error},
        title style={font=\scriptsize},
        title={SCLDMD},
        zmax=0.0008,
        label style={font=\scriptsize},
        ylabel style={yshift=0.3cm,xshift=0.25cm},
        xlabel style={yshift=0.3cm,xshift=-0.25cm},
        tick label style={font=\scriptsize},
        view={-35}{25},
        enlarge y limits=0,
        enlarge x limits=0,
        enlarge z limits=0.01]
    \addplot3[surf, mesh/rows=9, shader=interp] table [] {data/DuffingSCLDMDVectorFieldError.dat};
    \end{axis}
    \end{tikzpicture}
    \begin{tikzpicture}
    \begin{axis}[
        width=0.5\columnwidth,
        colormap/viridis,
        xlabel={$x_1$},
        ylabel={$x_2$},
        zlabel={Relative Error},
        title style={font=\scriptsize},
        title={CLDMD},
        zmax=0.0008,
        label style={font=\scriptsize},
        ylabel style={yshift=0.3cm,xshift=0.25cm},
        xlabel style={yshift=0.3cm,xshift=-0.25cm},
        tick label style={font=\scriptsize},
        view={-35}{25},
        enlarge y limits=0,
        enlarge x limits=0,
        enlarge z limits=0.01]
    \addplot3[surf, mesh/rows=9, shader=interp] table [] {data/DuffingCLDMDVectorFieldError.dat};
    \end{axis}
    \end{tikzpicture}
    \caption{Relative error $\frac{\Vert F_{\mu}(x) - \hat{F}_{\mu}(x)\Vert_2}{\max_{x\in [-2,2]\times[-2,2]}\Vert F_{\mu}(x)\Vert_2}$ in the estimation of the vector field $F_\mu$ of the controlled Duffing oscillator as a function of $x$, obtained using SCLDMD (left) and CLDMD (right).}
    \label{fig:vector_field_error}
\end{figure}

\subsubsection{Indirect Closed-loop Response Prediction}
The closed loop response can be predicted using either SCLDMD or CLDMD by numerically solving the initial value problems in \eqref{eq:convergent_closed_loop_model} and \eqref{eq:vector_field_reconstruction}, respectively, starting from the desired initial condition.
Fig. \ref{fig:Duffing_Indirect_Reconstruction} shows the prediction error resulting from this indirect approach, starting from $x_0 = \begin{bmatrix} 2 & -2 \end{bmatrix}^{\top}$. The results in Fig. \ref{fig:Duffing_Indirect_Reconstruction} indicate that both the CLDMD and the SCLDMD methods, when coupled with indirect prediction, accurately predict the desired closed-loop trajectory.
\begin{figure}[ht]
    \centering
    \begin{tikzpicture}    
        \begin{axis}[
            xlabel={Time [s]},
            title style={font=\scriptsize},
            title={SCLDMD},
            legend pos = south east,
            legend style={nodes={scale=0.5, transform shape}},
            enlarge y limits=0.05,
            enlarge x limits=0,
            ymax = 0.0005,
            ymin = -0.0007,
            height = 0.4\columnwidth,
            width = 0.5\columnwidth,
            label style={font=\scriptsize},
            tick label style={font=\scriptsize}
        ]
            \pgfplotsinvokeforeach{1,2}{
                \addplot+ [thick, mark=none] table [x index=0, y index=#1]{data/DuffingSCLDMDError.dat};
            }
            \legend{$x_{1}(t)-\hat{x}_1(t)$,$x_{2}(t)-\hat{x}_{2}(t)$}
        \end{axis}
    \end{tikzpicture}
    \begin{tikzpicture}
        \begin{axis}[
            xlabel={Time [s]},
            title style={font=\scriptsize},
            title={CLDMD},
            legend pos = south east,
            legend style={nodes={scale=0.5, transform shape}},
            enlarge y limits=0.05,
            enlarge x limits=0,
            ymax = 0.0005,
            ymin = -0.0007,
            height = 0.4\columnwidth,
            width = 0.5\columnwidth,
            label style={font=\scriptsize},
            tick label style={font=\scriptsize}
        ]
            \pgfplotsinvokeforeach{1,2}{
                \addplot+ [thick, mark=none] table [x index=0, y index=#1]{data/DuffingCLDMDError.dat};
            }
            \legend{$x_{1}(t)-\hat{x}_1(t)$,$x_{2}(t)-\hat{x}_{2}(t)$}
        \end{axis}
    \end{tikzpicture}
    \caption{Error between predicted and true trajectories of the controlled duffing oscillator for the experiment in Section \ref{subsec:duffing}. The figure on the left is obtained using SCLDMD indirect prediction by solving \eqref{eq:convergent_closed_loop_model} and the figure on the right is obtained using CLDMD indirect prediction by solving \eqref{eq:vector_field_reconstruction}, both using the MATLAB\textsuperscript{\textregistered} \texttt{ode45} solver.}
\label{fig:Duffing_Indirect_Reconstruction}
\end{figure}

\subsubsection{Direct Closed-loop Response Prediction}
The CLDMD method can also be used to predict the behavior of the closed-loop system starting from a given initial condition, and under the given feedback controller. Direct reconstruction is implemented using \eqref{eq:finite_spectral_reconstruction}. Fig. \ref{fig:Duffing_Direct_Reconstruction} shows the true and the predicted trajectories starting from the initial condition $x_0 = \begin{bmatrix} 2 & -2 \end{bmatrix}^{\top}$. The predicted trajectory is denoted by $\hat{x}$. The results in Fig. \ref{fig:Duffing_Direct_Reconstruction} indicate that the CLDMD method, when coupled with direct prediction, fails to obtain accurate prediction of the closed-loop trajectories.
\begin{figure}[ht]
    \centering
    \begin{tikzpicture}
        \begin{axis}[
            xlabel={Time [s]},
            legend pos = south east,
            legend style={nodes={scale=0.5, transform shape}},
            enlarge y limits=0.05,
            enlarge x limits=0,
            height = 0.4\columnwidth,
            width = 0.5\columnwidth,
            label style={font=\scriptsize},
            tick label style={font=\scriptsize}
        ]
            \pgfplotsinvokeforeach{1,2}{
                \addplot+ [thick, mark=none] table [x index=0, y index=#1]{data/DuffingCLDMDReconstructionDirect.dat};}
            \pgfplotsset{cycle list shift=2}
            \pgfplotsinvokeforeach{3,4}{
                \addplot+ [thick, dashed, mark=none] table [x index=0, y index=#1]{data/DuffingCLDMDReconstructionDirect.dat};}
             \legend{$x_{1}(t)$,$x_{2}(t)$,$\hat{x}_1(t)$,$\hat{x}_{2}(t)$}
        \end{axis}
    \end{tikzpicture}
    \begin{tikzpicture}
        \begin{axis}[
            xlabel={Time [s]},
            legend pos = south east,
            legend style={nodes={scale=0.5, transform shape}},
            enlarge y limits=0.05,
            enlarge x limits=0,
            height = 0.4\columnwidth,
            width = 0.5\columnwidth,
            label style={font=\scriptsize},
            tick label style={font=\scriptsize}
        ]
            \pgfplotsinvokeforeach{1,2}{
                \addplot+ [thick, mark=none] table [x index=0, y index=#1]{data/DuffingCLDMDErrorDirect.dat};}
            \legend{$x_{1}(t)-\hat{x}_1(t)$,$x_{2}(t)-\hat{x}_{2}(t)$}
        \end{axis}
    \end{tikzpicture}
    \caption{Predicted and true trajectories (left) and the corresponding prediction errors (right) of the controlled duffing oscillator for the experiment in Section \ref{subsec:duffing}. This result is obtained using CLDMD direct prediction \eqref{eq:finite_spectral_reconstruction} with kernel parameter  $\tilde{\rho} = 1e8$.\label{fig:Duffing_Direct_Reconstruction}}
\end{figure}

\subsection{Two-link Robot Manipulator\label{subsec:2link}}
This experiment concerns a planar two-link robot manipulator described by Euler-Lagrange dynamics
\begin{equation*}
    M(q)\ddot{q}+V_{m}(q,\dot{q})\dot{q}+F(\dot{q}) = \tau,
\end{equation*}
where $q=(q_{1}\:\:q_{2})^{\top}\in\mathbb{R}^2$ and $\dot{q}=(\dot{q}_{1}\:\:\dot{q}_{2})^{\top}$ are the angular positions ($\si{\radian}$) and angular velocities ($\si{\radian\per\second}$) of the two links, respectively, $\tau=\left(\tau_{1}\:\:\tau_{2}\right)^{\top}$ is the torque ($\si{\newton\meter}$) produced by the motors that drive the joints, $M(q)$ is the inertia matrix, and $V_{m}(q,\dot{q})$ is the centripetal-Coriolis matrix, defined as
\begin{equation*}
	M\left(q\right)\coloneqq\begin{bmatrix}
		p_{1}+2p_{3}c_{2}\left(q\right) & p_{2}+p_{3}c_{2}\left(q\right)\\
		p_{2}+p_{3}c_{2}\left(q\right) & p_{2}
	\end{bmatrix},\quad\text{and}\quad V_{m} \left(q,\dot{q}\right)=\begin{bmatrix}
		p_{3}s_{2}\left(q\right)\dot{q}_{2} & -p_{3}s_{2}\left(q\right)\left(\dot{q}_{1}+\dot{q}_{2}\right)\\
		p_{3}s_{2}\left(q\right)\dot{q}_{1} & 0
	\end{bmatrix},
\end{equation*}
where $p_{1}=\SI{3.473}{\kilogram\meter\squared}$, $p_{2}=\SI{0.196}{\kilogram\meter\squared}$, $p_{3}=\SI{0.242}{\kilogram\meter\squared}$, $c_{2}\left(q\right)=\cos(q_{2})$, $s_{2}\left(q\right)=\sin(q_{2}),$ and $F(\dot{q})=\begin{bmatrix}f_{d1}\dot{q}_{1}+f_{s1}\tanh(\dot{q}_{1}) & f_{d2}\dot{q}_{2}+f_{s2}\tanh(\dot{q}_{2})\end{bmatrix}^{\top}$ is the model for friction, where $f_{d1}=\SI{5.3}{\kilogram\meter\squared\per\second}$, $f_{d2}=\SI{1.1}{\kilogram\meter\squared\per\second}$, 
$f_{s1}=\SI{8.45}{\kilogram\meter\squared\per\second}$, and $f_{s2}=\SI{2.35}{\kilogram\meter\squared\per\second}$. The model can be expressed in the form $\dot{x} = f(x) + g(x)u$ with $x = \begin{bmatrix}
q^{\top} & \dot{q}^{\top}
\end{bmatrix}^{\top}$, $u = \tau$, $f(x) = \begin{bmatrix}
\dot{q}^{\top} & \left(M^{-1}(q)(-V_{m}(q,\dot{q})\dot{q}+F(\dot{q}))\right)^{\top}
\end{bmatrix}$, and $g(x) = \begin{bmatrix}
0_{2\times 2} & (M^{-1}(q))^{\top}
\end{bmatrix}^{\top}$, where $0_{2\times 2}$ denotes a $2\times2$ matrix of zeros.

A total of 200 open-loop trajectories of the manipulator are generated using the MATLAB\textsuperscript{\textregistered} \texttt{ode45} solver, starting from initial conditions selected to fill a hypercube of side $1$, centered at the origin of the state space, $\mathbb{R}^4$, using a Halton sequence. The control signal used for trajectory generation is of the form $u = \begin{bmatrix} u_1 & u_2 \end{bmatrix}^{\top}$ with $u_j(t) = \sum_{i=1}^{15} b_{j,i}\sin(\omega_{j,i} t + \varphi_{j,i})$, for $ j = 1,2 $, where the magnitudes $b_{j,i}$, the frequencies $\omega_{j,i}$, and the phase differences $\varphi_{j,i}$ are generated randomly from a uniform distribution on the interval $[-1,1]$. All trajectories are recorded over a duration of $1$\si{\second}, and are sampled at a frequency of $10$ \si{\hertz}.

The trajectories are then utilized to predict the behavior of the oscillator under the state feedback controller $\mu(x) = \begin{bmatrix} -5 & -5 \\ -15 & -15 \end{bmatrix}x$, starting from $x_0 = \begin{bmatrix}1&-1&1&-1\end{bmatrix}^{\top}$. CLDMD is implemented using the exponential dot product reproducing kernel with parameter $10$ and a diagonal kernel given by $ K = \mathrm{diag}\begin{bmatrix} \tilde{K}_{10} & \tilde{K}_{10} & \tilde{K}_{10} \end{bmatrix} $. SCLDMD is implemented using $\tilde{K}_r = \tilde{K}_{5} $, $ K = \mathrm{diag}\begin{bmatrix} \tilde{K}_{10} & \tilde{K}_{10} & \tilde{K}_{10} \end{bmatrix} $, and $\tilde{K}_d = \tilde{K}_{15} $. Gram matrices are regularized as described in Remark \ref{rem:rank_deficient} using regularization coefficients $\epsilon = \tilde{\epsilon} = 1e-3$. Simpson's 1/3 rule is used to compute the integrals involved in algorithms \ref{alg:SCLDMD} and \ref{alg:CLDMD}. Since the vector field is now a function of $4$ variables in each dimension, direct visualization of the true and approximate vector fields is not possible. However, the reconstruction accuracy may be indirectly gauged through indirect prediction of trajectories of the system.

Fig. \ref{fig:TwoLink_Direct_Prediction} shows the true and the predicted trajectories using the direct reconstruction method,  implemented using \eqref{eq:finite_spectral_reconstruction}. The results in Fig. \ref{fig:TwoLink_Direct_Prediction} indicate that the CLDMD method, when coupled with indirect prediction, is able to predict the desired closed-loop trajectory much better in this experiment than the Duffing oscillator experiment in Fig. \ref{fig:Duffing_Direct_Reconstruction}.
\begin{figure}
    \centering
    \begin{tikzpicture}
        \begin{axis}[
            xlabel={Time [s]},
            legend style={
                nodes={scale=0.5,
                transform shape},
                at={(0.6,0.5)},anchor=west
            },
            enlarge y limits=0.05,
            enlarge x limits=0,
            height = 0.4\columnwidth,
            width = 0.5\columnwidth,
            label style={font=\scriptsize},
            tick label style={font=\scriptsize}
        ]
            \pgfplotsinvokeforeach{1,...,4}{
                \addplot+ [thick, mark=none] table [x index=0, y index=#1]{data/2LinkCLDMDReconstructionDirect.dat};}
            \pgfplotsinvokeforeach{5,...,8}{
                \addplot+ [thick, dashed, mark=none] table [x index=0, y index=#1]{data/2LinkCLDMDReconstructionDirect.dat};}
             \legend{$x_{1}(t)$,$x_{2}(t)$,$x_{3}(t)$,$x_{4}(t)$,$\hat{x}_1(t)$,$\hat{x}_{2}(t)$,$\hat{x}_3(t)$,$\hat{x}_{4}(t)$}
        \end{axis}
    \end{tikzpicture}
    \begin{tikzpicture}
        \begin{axis}[
            xlabel={Time [s]},
            legend pos = north east,
            legend style={nodes={scale=0.5, transform shape}},
            enlarge y limits=0.05,
            enlarge x limits=0,
            height = 0.4\columnwidth,
            width = 0.5\columnwidth,
            label style={font=\scriptsize},
            tick label style={font=\scriptsize}
        ]
            \pgfplotsinvokeforeach{1,...,4}{
                \addplot+ [thick, mark=none] table [x index=0, y index=#1]{data/2LinkCLDMDErrorDirect.dat};}
            \legend{$x_{1}(t)-\hat{x}_1(t)$,$x_{2}(t)-\hat{x}_{2}(t)$,$x_{3}(t)-\hat{x}_{3}(t)$,$x_{4}(t)-\hat{x}_{4}(t)$}
        \end{axis}
    \end{tikzpicture}
    \caption{Predicted and true trajectories (left) and the corresponding prediction errors (right) of the 2-link robot manipulator for the experiment in Section \ref{subsec:2link}. This result is obtained using CLDMD direct prediction \eqref{eq:finite_spectral_reconstruction} with kernel parameter $\tilde{\rho} = 1e5$ and regularization parameter $\tilde{\epsilon}=\epsilon=1e-7$.}
    \label{fig:TwoLink_Direct_Prediction}
\end{figure}
Fig. \ref{fig:TwoLink_Indirect_prediction} shows the predicted trajectories and the prediction error resulting from the indirect approach. The results in Fig. \ref{fig:TwoLink_Indirect_prediction} indicate that both the CLDMD and the SCLDMD methods, when coupled with indirect prediction, accurately predict the desired closed-loop trajectory.
\begin{figure}
    \centering
    \begin{tikzpicture}    
        \begin{axis}[
            xlabel={Time [s]},
            legend pos = north east,
            legend style={nodes={scale=0.5, transform shape}},
            enlarge y limits=0.05,
            enlarge x limits=0,
            title style={font=\scriptsize},
            title={SCLDMD},
            ymax = 0.025,
            ymin = -0.022,
            height = 0.4\columnwidth,
            width = 0.5\columnwidth,
            label style={font=\scriptsize},
            tick label style={font=\scriptsize}
        ]
            \pgfplotsinvokeforeach{1,...,4}{
                \addplot+ [thick, mark=none] table [x index=0, y index=#1]{data/2LinkSCLDMDError.dat};
            }
            \legend{$x_{1}(t)-\hat{x}_1(t)$,$x_{2}(t)-\hat{x}_{2}(t)$,$x_{3}(t)-\hat{x}_{3}(t)$,$x_{4}(t)-\hat{x}_{4}(t)$}
        \end{axis}
    \end{tikzpicture}
    \begin{tikzpicture}
        \begin{axis}[
            xlabel={Time [s]},
            legend pos = north east,
            legend style={nodes={scale=0.5, transform shape}},
            enlarge y limits=0.05,
            title style={font=\scriptsize},
            title={CLDMD},
            enlarge x limits=0,
            ymax = 0.025,
            ymin = -0.022,
            height = 0.4\columnwidth,
            width = 0.5\columnwidth,
            label style={font=\scriptsize},
            tick label style={font=\scriptsize}
        ]
            \pgfplotsinvokeforeach{1,...,4}{
                \addplot+ [thick, mark=none] table [x index=0, y index=#1]{data/2LinkCLDMDError.dat};
            }
            \legend{$x_{1}(t)-\hat{x}_1(t)$,$x_{2}(t)-\hat{x}_{2}(t)$,$x_{3}(t)-\hat{x}_{3}(t)$,$x_{4}(t)-\hat{x}_{4}(t)$}
        \end{axis}
    \end{tikzpicture}
    \caption{Error between the predicted and the true trajectories of the 2-link manipulator for the experiment in Section \ref{subsec:2link}. The figure on the left is obtained using SCLDMD indirect prediction by solving \eqref{eq:convergent_closed_loop_model} and the figure on the right is obtained using CLDMD indirect prediction by solving \eqref{eq:vector_field_reconstruction}, both using the MATLAB\textsuperscript{\textregistered} \texttt{ode45} solver.}
    \label{fig:TwoLink_Indirect_prediction}
\end{figure}

\section{Discussion}\label{sec:discussion}
As is evident from Fig. \ref{fig:vector_field_error}, the methods developed in algorithms \ref{alg:SCLDMD} and \ref{alg:CLDMD} can effectively utilize data collected under open-loop control signals to construct the closed-loop vector field under a given feedback policy. 

Figures \ref{fig:Duffing_Direct_Reconstruction} and \ref{fig:Duffing_Indirect_Reconstruction} indicate that the prediction error for the highly nonlinear controlled Duffing oscillator is significantly higher in direct prediction as compared to indirect prediction. We postulate that this is due to the linear nature of the model in \eqref{eq:finite_spectral_reconstruction}. Since \eqref{eq:finite_spectral_reconstruction} is a solution of a system of linear ordinary differential equations, the resulting reconstruction diverges quickly from the trajectories of the nonlinear model. On the other hand, we postulate that due to the presence of the eigenfunctions in \eqref{eq:vector_field_reconstruction}, the model used in the indirect approach includes nonlinear effects and as a result, generates a better prediction. When the nonlinearities in the original system are mild, like the trigonometric nonlinearities in the two link robot model, the predictions from the direct and the indirect method are close, as seen in figures \ref{fig:TwoLink_Direct_Prediction} and \ref{fig:TwoLink_Indirect_prediction}.

The theory and the computations that support the developed algorithms require data-richness. In the convergence proofs, data-richness manifests as the density of the kernel differences, the occupation kernels, and the control occupation kernels in their respective RKHSs. In the computations, data-richness is required for the Gram matrices $G_r$ and $G_\beta$ of the occupation kernels and the control occupation kernels, respectively to be invertible. It is shown in \cite{SCC.Russo.Kamalapurkar.ea2022} that for Gaussian radial basis function reproducing kernels, the rank of the occupation kernel Gram matrix $G_r$ can be characterized using the so-called trajectory separation distance. Roughly, the trajectory separation distance is the largest radius $q$ such that when all trajectories are inflated to tubes of radius $q$, the resulting tubes are disjoint. Since multiple, shorter trajectories, starting from initial conditions that are well-separated, would generally result in better separation distances, such a dataset would be preferred. However, if the trajectories are too short, then the matrix $D$ of trajectory endpoint differences can reduce to a zero matrix, resulting in poor performance. Obtaining similar results for other reproducing kernels and for characterization of the rank of the control occupation kernel Gram matrix $G_\beta$ is a topic for future research.

The condition number of  $G_r$ and $G_\beta$ depends not only on the trajectories but also on the selected reproducing kernels. For example, the condition number of Gram matrices corresponding to Gaussian radial basis functions, given as $\tilde K(x,y) = \exp(-\frac{1}{\tilde{\rho}}\|x-y\|_2^2)$ for $\tilde{\rho} > 0$ is larger for larger $\tilde{\rho}$. However, large $\tilde{\rho}$ values correspond to faster convergence of interpolation problems within the native space of the kernel (cf. \cite{SCC.Fasshauer2007}). Data-richness conditions similar to the persistence of excitation (PE) condition in adaptive control that relate the trajectories and the kernels can potentially be formulated to ensure a well-conditioned  $G_r$ and $G_\beta$, however, such formulation is a topic for future research.

Numerical experiments indicate that while direct trajectory reconstruction can be poor for systems with severe nonlinearities, the developed techniques generate accurate estimates of the closed-loop vector field from data. Unlike traditional system identification techniques, the algorithms developed in this paper do not require careful selection of basis functions. While the implementation can be done using any universal kernels, careful tuning of the kernel parameter is often necessary.

\section{Conclusion\label{sec:Conclusion}}
In this paper, a novel operator-theoretic framework is developed for the study of controlled nonlinear systems. The framework utilizes RKHSs, where feedback-controlled nonlinear systems are expressed using a composition of infinite dimensional multiplication operator and an infinite dimensional control Liouville operator. A provably convergent finite-rank representation of the composition, that utilizes trajectories of a system, observed under open-loop control inputs $u_j$, is developed. Eigendecomposition and SVD of the finite-rank representation is utilized to predict the behavior of the system response to a query feedback controller, $\mu$. The same dataset can be used to predict the system behavior in response to a multitude of query feedback controllers. 

To the best of our knowledge, this paper, along with the conference paper \cite{SCC.Rosenfeld.Kamalapurkar2021}, are the first to study spectral decomposition of continuous-time feedback-controlled nonlinear systems in a provably convergent manner. While this paper focuses solely on system identification, the uniform convergence guarantees established by Corollary \ref{cor:uniform-convergence}, makes the developed modeling technique an attractive candidate for use in a variety of applications, including, but not limited to, data-driven control synthesis and data-driven analysis and validation of feedback controllers.
\bibliographystyle{IEEEtran}
\bibliography{scc,sccmaster,scctemp}
\end{document}